\documentclass[11pt,a4paper]{article}
\usepackage{graphicx}
\usepackage{hyperref}
\hypersetup{
    colorlinks=true,
    linkcolor=blue,
    filecolor=magenta,      
    urlcolor=cyan,
}
\usepackage[font=sf, labelfont={sf,bf}, margin=1cm]{caption}
\usepackage{subcaption}

\usepackage{xcolor}
\usepackage{comment}
\usepackage[utf8]{inputenc}
\usepackage{titlesec}
\titleformat*{\subsection}{\bfseries}
\titleformat*{\subsubsection}{\it\bfseries}

\titleformat{\paragraph}
{\normalfont\normalsize\bfseries}{\theparagraph}{1em}{}
\usepackage{setspace}
\usepackage{amsmath}
\usepackage{amsthm}
\usepackage{amsfonts}
\usepackage{mathtools}
\usepackage{bm}

\usepackage{algorithm}
\usepackage[]{algpseudocode}
\makeatletter

\usepackage{longtable}
    \usepackage{multirow}

\usepackage{siunitx,booktabs}
\usepackage{tablefootnote}

\numberwithin{Theorem}{section}
\numberwithin{Definition}{section}
\numberwithin{Lemma}{section}
\numberwithin{Algorithm}{section}
\numberwithin{equation}{section}

\newtheorem{theorem}{Theorem}[section]

\newtheorem{lemma}[theorem]{Lemma} 
\newtheorem{assumption}{Assumption}
\newtheorem{premise}{Premise}
\newtheorem{definition}{Definition}

\usepackage{geometry}
\begin{document}

	\title{An Interior Point-Proximal Method of Multipliers for Convex Quadratic Programming}
\author{Spyridon Pougkakiotis   \and  Jacek Gondzio}

\maketitle


%\institute{S. Pougkakiotis \at University of Edinburgh
 %              \\
 %             \email{s.pougkakiotis@ed.ac.uk}           %  \\
 %          \and
 %          J. Gondzio \at University of Edinburgh\\
 %             \email{j.gondzio@ed.ac.uk}
%}


\begin{abstract}
\noindent In this paper we combine an infeasible Interior Point Method (IPM) with the Proximal Method of Multipliers (PMM). The resulting algorithm (IP-PMM) is interpreted as a primal-dual regularized IPM, suitable for solving linearly constrained convex quadratic programming problems. We apply few iterations of the interior point method to each sub-problem of the proximal method of multipliers. Once a satisfactory solution of the PMM sub-problem is found, we update the PMM parameters, form a new IPM neighbourhood and repeat this process. Given this framework, we prove polynomial complexity of the algorithm, under standard assumptions. To our knowledge, this is the first polynomial complexity result for a primal-dual regularized IPM. The algorithm is guided by the use of a single penalty parameter; that of the logarithmic barrier. In other words, we show that IP-PMM inherits the polynomial complexity of IPMs, as well as the strict convexity of the PMM sub-problems. The updates of the penalty parameter are controlled by IPM, and hence are well-tuned, and do not depend on the problem solved. Furthermore, we study the behavior of the method when it is applied to an infeasible problem, and identify a necessary condition for infeasibility. The latter is used to construct an infeasibility detection mechanism. Subsequently, we provide a robust implementation of the presented algorithm and test it over a set of small to large scale linear and convex quadratic programming problems. The numerical results demonstrate the benefits of using regularization in IPMs as well as the reliability of the method.
%\keywords{Interior Point Methods; Proximal Point Methods; Regularized Primal-Dual Methods; Convex Quadratic Programming}
\end{abstract}

\section{Introduction} \label{intro}
\subsection{Primal-Dual Pair of Convex Quadratic Programming Problems}
\par In this paper, we consider the following primal-dual pair of linearly constrained convex quadratic programming problems, in the standard form:
\begin{equation} \label{non-regularized primal} \tag{P}
\text{min}_{x} \ \big( c^Tx + \frac{1}{2}x^T Q x \big), \ \ \text{s.t.}  \  Ax = b,   \ x \geq 0, 
\end{equation}
\begin{equation} \label{non-regularized dual} \tag{D}
\text{max}_{x,y,z}  \ \big(b^Ty - \frac{1}{2}x^T Q x\big), \ \ \text{s.t.}\  -Qx + A^Ty + z = c,\ z\geq 0,
\end{equation}
\noindent with $c,x,z \in \mathbb{R}^n$, $b,y \in \mathbb{R}^m$, $A \in \mathbb{R}^{m\times n}$, and symmetric positive semi-definite $Q \in \mathbb{R}^{n \times n}$ (i.e. $Q \succeq 0$). Without loss of generality we assume that $m \leq n$. Duality, between the problems (\ref{non-regularized primal})-(\ref{non-regularized dual}), arises by introducing the Lagrangian function of the primal, using $y \in \mathbb{R}^m$ and $z \in \mathbb{R}^n,\ z \geq 0$, as the Lagrange multipliers for the equality and inequality constraints, respectively. Hence, we obtain:
\begin{equation} \label{standard Lagrangian} 
\mathcal{L}(x,y,z) = c^T x + \frac{1}{2}x^T Q x - y^T(Ax - b) - z^T x.
\end{equation}

\par Using the Lagrangian function, one can formulate the first-order optimality conditions (known as Karush--Kuhn--Tucker (KKT) conditions) for this primal-dual pair. In particular, we define the vector $w = (x^T,y^T,z^T)^T$, and compute the gradient of $\mathcal{L}(w)$. Using $\nabla \mathcal{L}(w)$,  as well as the complementarity conditions, we may define a function $F(w): \mathbb{R}^{2n + m} \mapsto \mathbb{R}^{2n+m}$, using which, we write the KKT conditions as follows:

\begin{equation} \label{non-regularized FOC}
F(w) = \begin{bmatrix}
c + Qx - A^T y - z \\
Ax - b\\
XZe_n\\
\end{bmatrix}
= 
\begin{bmatrix}
0\\
0\\
0\\
\end{bmatrix}, \qquad (x,z) \geq (0,0),
\end{equation} 
\noindent where $e_n$ denotes the vector of ones of size $n$, while $X,\ Z \in \mathbb{R}^{n \times n}$ denote the diagonal matrices satisfying $X^{ii} = x^i$ and $Z^{ii} = z^i$, $\forall\ i \in \{1,\cdots,n\}$. For the rest of this manuscript, superscripts will denote the components of a vector/matrix. For simplicity of exposition, when referring to convex quadratic programming problems, we implicitly assume that the problems are linearly constrained. If both (\ref{non-regularized primal}) and (\ref{non-regularized dual}) are feasible problems, it can easily be verified that there exists an optimal primal-dual triple $(x,y,z)$, satisfying the KKT optimality conditions of this primal-dual pair (see for example in \cite[Prop. 2.3.4]{book_2}). 

\subsection{A Primal-Dual Interior Point Method}
\par Primal-dual Interior Point Methods (IPMs) are popular for simultaneously solving (\ref{non-regularized primal}) and (\ref{non-regularized dual}), iteratively. As indicated in the name, primal-dual IPMs act on both primal and dual variables. There are numerous variants of IPMs and the reader is referred to \cite{paper_16} for an extended literature review. In this paper, an infeasible primal-dual IPM is presented. Such methods are called infeasible because they allow intermediate iterates of the method to be infeasible for (\ref{non-regularized primal})-(\ref{non-regularized dual}), in contrast to feasible IPMs, which require intermediate iterates to be strictly feasible.
\par Interior point methods handle the non-negativity constraints of the problems with logarithmic barriers in the objective. That is, at each iteration, we choose a \textit{barrier parameter} $\mu$ and form the logarithmic \textit{barrier primal-dual pair}:
\begin{equation} \label{Primal Barrier Problem}
\begin{split}
\text{min}_{x} \ \big( c^Tx + \frac{1}{2}x^T Q x - \mu \sum_{j=1}^n \ln x^j \big) \ \ \text{s.t.} \  Ax = b,  \\ 
\end{split}
\end{equation}
\begin{equation} \label{Dual Barrier Problem}
\begin{split}
\text{max}_{x,y,z} \ \big(\ b^Ty - \frac{1}{2}x^T Q x + \mu \sum_{j=1}^n \ln  z^j \big) \ \ \text{s.t.} \  -Qx + A^Ty + z = c,\\
\end{split}
\end{equation}
\noindent in which non-negativity constraints $x>0$ and $z>0$ are implicit. We form the KKT optimality conditions of the pair (\ref{Primal Barrier Problem})-(\ref{Dual Barrier Problem}), by introducing the Lagrangian of the primal barrier problem:

\begin{equation*}
\mathcal{L}^{\text{IPM}}_{\mu}(x,y) = c^T x + \frac{1}{2} x^T Q x -y^T (Ax - b) - \mu \sum_{i=1}^n \ln x^j.
\end{equation*} 

\noindent Equating the gradient of the previous function to zero, gives the following conditions:

\begin{equation*}
\begin{split}
\nabla_x \mathcal{L}^{\text{IPM}}_{\mu} (x,y) = &\ c + Qx - A^T y -\mu X^{-1} e_n  = 0, \\
\nabla_y \mathcal{L}^{\text{IPM}}_{\mu} (x,y) = &\ Ax - b   = 0.\\
\end{split}
\end{equation*}

\noindent Using the variable $z = \mu X^{-1}e_n$, the final conditions read as follows:

\begin{equation*}
\begin{split}
 c + Qx - A^T y - z  & = 0\\
 Ax - b   & = 0\\
XZe_n - \mu e_n & = 0.
\end{split}
\end{equation*}

\par At each IPM iteration, we want to approximately solve the following mildly non-linear system:

\begin{equation} \label{IPM FOC}
F^{\text{IPM}}_{\sigma,\mu}(w) = \begin{bmatrix}
c + Qx - A^Ty - z \\
b - Ax \\
XZe_n-\sigma \mu e_n\\
\end{bmatrix} = \begin{bmatrix}
0\\
0\\
0
\end{bmatrix},
\end{equation}
\noindent where $F^{\text{IPM}}_{\sigma,\mu}(w) = 0$ is a slightly perturbed form of the previously presented optimality conditions. In particular, $\sigma \in (0,1)$ is a \textit{centering parameter} which determines how fast $\mu$ will be forced to decrease at the next IPM iteration. For $\sigma = 1$, we recover the barrier optimality conditions, while for $\sigma = 0$ we recover the initial problems' optimality conditions (\ref{non-regularized FOC}). The efficiency of the method depends heavily on the choice of $\sigma$. In fact, various improvements of the traditional IPM schemes have been proposed in the literature, which solve the previous system for multiple carefully chosen values of the centering parameter $\sigma$ and of the right hand side, at each IPM iteration. These are the so called \textit{predictor-corrector} schemes, proposed for the first time in \cite{paper_6}. Various extensions of such methods have been proposed and analyzed in the literature (e.g. see \cite{paper_28,paper_29} and references therein). However, for simplicity of exposition, we will follow the traditional approach; $\sigma$ is chosen heuristically at each iteration, and the previous system is solved only once.

\par In order to approximately solve the system $F^{\text{IPM}}_{\sigma,\mu}(w) = 0$ for each value of $\mu$, the most common approach is to apply the Newton method. Newton method is favored for systems of this form, due to the \textit{self-concordance} of the function $\ln(\cdot)$. For more details on this subject, the interested reader is referred to \cite[Chapter 2]{book_3}. At the beginning of the $k$-th iteration of the IPM, for $k \geq 0$, we have available an iterate $w_k = (x_k^T,y_k^T,z_k^T)^T$, and a barrier parameter $\mu_k$, defined as $\mu_k = \frac{x_k^Tz_k}{n}$. We choose a value for the centering parameter $\sigma_k \in (0,1)$ and form the Jacobian of $F^{\text{IPM}}_{\sigma_k, \mu_k}(\cdot)$, evaluated at $w_k$. Then, the Newton direction is determined by solving the following system of equations:

\begin{equation} \label{non-regularized Newton System}
\begin{bmatrix} 
-Q &  A^T & I\\
A  & 0 & 0\\
Z_k &  0  &X_k 
\end{bmatrix}
\begin{bmatrix}
\Delta x_k\\
\Delta y_k\\
\Delta z_k
\end{bmatrix}
= -
\begin{bmatrix}
-c - Qx_k  + A^T y_k + z_k\\
Ax_k  - b\\
 X_kZ_ke_n - \sigma_{k} \mu_k e_n
\end{bmatrix}.
\end{equation}
\par Notice that as $\mu_k \rightarrow 0$, the optimal solution of (\ref{Primal Barrier Problem})-(\ref{Dual Barrier Problem}) converges to the optimal solution of (\ref{non-regularized primal})-(\ref{non-regularized dual}). Polynomial convergence of such methods has been established multiple times in the literature (see for example \cite{book_4,paper_30} and the references therein). 

\subsection{Proximal Point Methods}
\subsubsection{Primal Proximal Point Method}

\par One possible method for solving the primal problem (\ref{non-regularized primal}), is the so called proximal point method. Given an arbitrary starting point $x_0$, the $k$-th iteration of the method is summarized by the following minimization problem:
\begin{equation*}
x_{k+1} = \text{arg}\min_x \{c^T x + \frac{1}{2}x^T Q x + \frac{\mu_k}{2}\|x - x_k\|_2^2,\ \ \text{s.t.} \ Ax = b,\ x \geq 0\},
\end{equation*}
\noindent where $\mu_k > 0$ is a non-increasing sequence of penalty parameters, and $x_k$ is the current estimate for an optimal solution of (\ref{non-regularized primal}) (the use of $\mu_k$ is not a mistake; as will become clearer in Section \ref{section Algorithmic Framework}, we intend to use the barrier parameter $\mu_k$ to control both the IPM and the proximal method). Notice that such an algorithm is not practical, since we have to solve a sequence of sub-problems of similar difficulty to that of (\ref{non-regularized primal}). Nevertheless, the proximal method contributes the $\mu_k I$ term to the Hessian of the objective, and hence the sub-problems are strongly convex. This method is known to achieve a linear convergence rate (and possibly superlinear if $\mu_k \rightarrow 0$ at a suitable rate), as shown in \cite{paper_13,paper_12}, among others, even in the case where the minimization sub-problems are solved approximately. Various extensions of this method have been proposed in the literature. For example, one could employ different penalty functions other than the typical 2-norm penalty presented previously (e.g. see \cite{paper_20,paper_14,paper_31}). 
\par Despite the  fact that such algorithms are not practical, they have served as a powerful theoretical tool and a basis for other important methods. For an extensive review of the applications of standard primal proximal point methods, we refer the reader to \cite{paper_21}.

\subsubsection{Dual Proximal Point Method}

\par It is a well-known fact, observed for the first time in \cite{paper_13}, that the application of the proximal point method on the dual problem, is equivalent to the application of the \textit{augmented Lagrangian method} on the primal problem, which was proposed for the first time in \cite{paper_32,paper_33}. In view of the previous fact, we will present the derivation of the augmented Lagrangian method for the primal problem (\ref{non-regularized primal}), while having in mind that an equivalent reformulation of the model results in the proximal point method for the dual (\ref{non-regularized dual}) (see \cite{paper_13,paper_2} or \cite[Chapter 3.4.4]{book_5}).

\par We start by defining the augmented Lagrangian function of (\ref{non-regularized primal}). Given an arbitrary starting guess for the dual variables $y_0$, the augmented Lagrangian function is defined at the $k$-th iteration as:

\begin{equation} \label{ALM Lagrangian}
\mathcal{L}^{\text{ALM}}_{\mu_k}(x;y_k) = c^Tx + \frac{1}{2}x^T Q x -y_k^T (Ax - b) + \frac{1}{2\mu_k}\|Ax-b\|_2^2,
\end{equation}
\noindent where $\mu_k > 0$ is a non-increasing sequence of penalty parameters and $y_k$ is the current estimate of an optimal Lagrange multiplier vector. The augmented Lagrangian method (ALM) is summarized below:
\begin{equation*}
\begin{split}
x_{k+1} = &\ \text{arg}\min_{x} \{ \mathcal{L}^{\text{ALM}}_{\mu_k}(x;y_k), \ \ \text{s.t.}\ x \geq 0\}, \\
y_{k+1} = &\ y_k - \frac{1}{\mu_k} (Ax_{k+1} - b).
\end{split}
\end{equation*}

\par Notice that the update of the Lagrange multiplier estimates can be interpreted as the application of the dual ascent method. A different type of update would be possible, however, the presented update is cheap and effective, due to the strong concavity of the proximal (``regularized") dual problem (since $\mu_k > 0$). Convergence and iteration complexity of such methods have been established multiple times (see for example \cite{paper_13,paper_34}). There is a vast literature on the subject of augmented Lagrangian methods. For a plethora of technical results and references, the reader is referred to \cite{book_6}. For convergence results of various extended versions of the method, the reader is referred to \cite{paper_31}.
\par It is worth noting here that a common issue, arising when using augmented Lagrangian methods, is that of the choice of the penalty parameter. To the authors' knowledge, an optimal tuning for this parameter is still unknown.
\subsubsection{Proximal Method of Multipliers}

\par In \cite{paper_13}, the author presented, for the first time, the \textit{proximal method of multipliers} (PMM). The method consists of applying both primal and dual proximal point methods for solving (\ref{non-regularized primal})-(\ref{non-regularized dual}). For an arbitrary starting point $(x_0,y_0)$, and using the already defined augmented Lagrangian function, given in (\ref{ALM Lagrangian}), PMM can be summarized by the following iteration:
\begin{equation} \label{PMM sub-problem}
\begin{split}
x_{k+1} =&\ \text{arg}\min_x \{ \mathcal{L}^{\text{ALM}}_{\mu_k}(x;y_k) + \frac{\mu_k}{2}\|x-x_k\|_2^2,\ \text{s.t.}\ x \geq 0\},\\
y_{k+1} = &\ y_k - \frac{1}{\mu_k} (Ax_{k+1} - b),
\end{split}
\end{equation}

\noindent where $\mu_k$ is a positive non-increasing penalty parameter. One can observe that at every iteration of the method, the primal problem that we have to solve is strongly convex, while its dual, strongly concave. As shown in \cite{paper_13}, the addition of the term $\mu_k \|x - x_k\|_2^2$, in the augmented Lagrangian method does not affect its convergence rate, while ensuring better numerical behaviour of its iterates. An extension of this algorithm, allowing for the use of more general penalties, can be found in \cite{paper_31}.

\par We can write \eqref{PMM sub-problem} equivalently by making use of the maximal monotone operator $T_{\mathcal{L}}: \mathbb{R}^{n+m} \mapsto \mathbb{R}^{n+m}$ associated to \eqref{non-regularized primal}-\eqref{non-regularized dual} (see \cite{paper_39,paper_13}), that is defined as:

\begin{equation} \label{Primal Dual Maximal Monotone Operator}
T_{\mathcal{L}}(x,y) = \{(u,v): v \in Qx + c - A^Ty + \partial \delta_+(x), u = Ax-b \},
\end{equation}
\noindent where $\delta_+(\cdot)$ is an indicator function defined as:
\begin{equation} \label{Indicator function}
\delta_+(x) = 
     \begin{cases}
       \infty &\quad\text{if } \exists\ j: x^j < 0, \\
       0 &\quad\text{otherwise.} \\ 
     \end{cases}
\end{equation}
\noindent and $\partial(\cdot)$ denotes the sub-differential of a function. In our case, we have that (see \cite[Section 23]{book_7}): 
\begin{equation*}
z \in \partial \delta_+(x) \Leftrightarrow \begin{cases}
     z^j = 0 & \text{if }  x^j > 0,\\
    z^j \leq 0 & \text{if } x^j = 0.                                             \end{cases},\ \forall\ j = \{1,\cdots,n\}.
\end{equation*}
\noindent By convention, if there exists a component $j$ such that $x_j < 0$, we have that $\partial \delta_+(x) = \{\emptyset\}$. Given a bounded pair $(x^*,y^*)$ such that: $0 \in T_{\mathcal{L}}(x^*,y^*)$, we can retrieve a vector $z^* \in \partial \delta_+(x^*)$, using which $(x^*,y^*,-z^*)$ is an optimal solution for \eqref{non-regularized primal}-\eqref{non-regularized dual}. By letting:
\begin{equation} \label{Primal Dual Proximal Operator}
P_k = (I_{(m+n)} + \frac{1}{\mu_k}T_{\mathcal{L}})^{-1},
\end{equation}
\noindent we can express \eqref{PMM sub-problem} as: 
\begin{equation} \label{PMM Operator Subproblem}
(x_{k+1},y_{k+1}) = P_k(x_k,y_k),
\end{equation}
\noindent and it can be shown that $P_k$ is single valued and non-expansive (see \cite{paper_39}).
\subsection{Regularization in Interior Point Methods}
\par In the context of interior point methods, it is often beneficial to include a regularization in order to improve the spectral properties of the system matrix in (\ref{non-regularized Newton System}). For example, notice that if the constraint matrix $A$ is rank-deficient, then the matrix in (\ref{non-regularized Newton System}) is not invertible. The latter can be immediately addressed by the introduction of a dual regularization, say $\delta > 0$, ensuring that $\text{rank}([A\ |\ \delta I_m] ) = m$. For a detailed analysis of the effect of dual regularization on such systems, the reader is referred to \cite{paper_15}. On the other hand, the most common and efficient approach in practice is to eliminate the variables $\Delta z$ from system (\ref{non-regularized Newton System}), and solve the following symmetric reduced (\textit{augmented}) system instead:

\begin{equation} \label{non-regularized Augmented System}
\begin{bmatrix} 
-(Q+\Theta_k^{-1}) &   A^T \\
A & 0
\end{bmatrix}
\begin{bmatrix}
\Delta x_k\\ 
\Delta y_k
\end{bmatrix}
= 
\begin{bmatrix}
c + Qx_k - A^T y_k - \sigma_{k} \mu_k X_k^{-1} e\\
b-Ax_k
\end{bmatrix},
\end{equation}
\noindent where $\Theta_k = X_k Z_k^{-1}$. Since $X_k z_k \rightarrow 0$ close to optimality, one can observe that the matrix $\Theta_k$ will contain some very large and some very small elements. Hence, the matrix in (\ref{non-regularized Augmented System}) will be increasingly ill-conditioned, as we approach optimality. The introduction of a primal regularization, say $\rho > 0$, can ensure that the matrix $Q+\Theta_k^{-1} + \rho I_n$ will have eigenvalues that are bounded away from zero, and hence a significantly better worst-case conditioning than that of $Q+\Theta_k^{-1}$. In other words, regularization is commonly employed in IPMs, as a means of improving robustness and numerical stability (see \cite{paper_1}). As we will discuss later, by introducing regularization in IPMs, one can also gain efficiency. This is because regularization transforms the matrix in (\ref{non-regularized Augmented System}) into a quasi-definite one. It is known that such matrices can be factorized efficiently (see \cite{paper_4}).
\par In view of the previous discussion, one can observe that a very natural way of introducing primal regularization to problem (\ref{non-regularized primal}), is through the application of the primal proximal point method. Similarly, dual regularization can be incorporated through the application of the dual proximal point method. This is a well-known fact. The authors in \cite{paper_1} presented a primal-dual regularized IPM, and interpreted this regularization as the application of the proximal point method. Consequently, the authors in \cite{paper_2} developed a primal-dual regularized IPM, which applies PMM to solve \eqref{non-regularized primal}, and a single IPM iteration is employed for approximating the solution of each PMM sub-problem. There, global convergence of the method was proved, under some assumptions. A variation of the method proposed in \cite{paper_2} is given in \cite{paper_35}, where general non-diagonal regularization matrices are employed, as a means of further improving factorization efficiency.
\par Similar ideas have been applied in the context of IPMs for general non-linear programming problems. For example, the authors in \cite{paper_17} presented an interior point method combined with the augmented Lagrangian method, and proved that, under some general assumptions, the method converges at a superlinear rate. Similar approaches can be found in \cite{paper_18,paper_36,paper_37}, among others. 
\par In this paper, we develop a path-following primal-dual regularized IPM for solving convex quadratic programming problems. The algorithm is obtained by applying one or a few iterations of the infeasible primal-dual interior point method in order to solve sub-problems arising from the primal-dual proximal point method. Under standard assumptions, we prove polynomial complexity of the algorithm and provide global convergence guarantees. To our knowledge, this is the first polynomial complexity result for a general primal-dual regularized IPM scheme. Notice that a complexity result is given for a primal regularized IPM for linear complementarity problems in \cite{paper_38}. However, the authors significantly alter the Newton directions, making their method very difficult to generalize and hard to achieve efficiency in practice. An important feature of the presented method is that it makes use of only one penalty parameter, that is the logarithmic penalty term. The aforementioned penalty has been extensively studied and understood, and as a result, IPMs achieve fast and reliable convergence in practice. This is not the case for the penalty parameter of the proximal methods. In other words, IP-PMM inherits the fast and reliable convergence properties of IPMs, as well as the strong convexity of the PMM sub-problems, hence improving the conditioning of the Newton system solved at each IPM iteration, while providing a reliable tuning for the penalty parameter, independently of the problem at hand. The proposed approach is implemented and its reliability is demonstrated through extensive experimentation. The implementation slightly deviates from the theory, however, most of the theoretical results are verified in practice. The main purpose of this paper is to provide a reliable method that can be used for solving general convex quadratic problems, without the need of pre-processing, or of previous knowledge about the problems. The implemented method is supported by a novel theoretical result, indicating that regularization alleviates various issues arising in IPMs, without affecting their important worst-case polynomial complexity. As a side-product of the theory, an implementable infeasibility detection mechanism is also derived and tested in practice.
\par Before completing this section, let us introduce some notation that is used in the rest of this paper. An iteration of the algorithm is denoted by $k \in \mathbb{N}$. Given an arbitrary matrix (vector) $A$ ($x$, respectively), $A_k$ ($x_k$) denotes that $A$ ($x$) depends on the iteration $k$. An optimal solution to the pair \eqref{non-regularized primal}-\eqref{non-regularized dual} will be denoted as $(x^*,y^*,z^*)$. Optimal solutions of different primal-dual pairs will be denoted using an appropriate subscript, in order to distinguish them. For example, we use the notation $(x_r^*,y_r^*,z_r^*)$, for representing an optimal solution for a PMM sub-problem. The subscript is employed for distinguishing the ``regularized" solution, from the solution of the initial problem, that is $(x^*,y^*,z^*)$. Any norm (semi-norm respectively) is denoted by $\| \cdot \|_{\nu}$, where $\nu$ is used to distinguish between different norms. For example, the 2-norm (Euclidean norm) will be denoted as $\|\cdot\|_2$. When a given scalar is assumed to be independent of $n$ and $m$, we mean that this quantity does not depend on the problem dimensions. Given two vectors of the same size $x$, $y$, $x \geq y$ denotes that the inequality holds component-wise.  Given two logical statements $T_1,\ T_2$, the condition $T_1 \wedge T_2$ is true only when both $T_1$ and $T_2$ are true.  Let an arbitrary matrix $A$ be given. The maximum (minimum) singular value of $A$ is denoted as $\eta_{\max}(A)$ ($\eta_{\min}(A)$). Similarly, if $A$ is square, its maximum eigenvalue is denoted as $\nu_{\max}(A)$. Given a set of indices, say $\mathcal{B} \subset \{1,\ldots,n\}$, and an arbitrary vector (or matrix) $x \in \mathbb{R}^n$  ($A \in \mathbb{R}^{n \times n}$), $x^{\mathcal{B}}$ ($A^{\mathcal{B}}$, respectively) denotes that sub-vector (sub-matrix) containing the elements of $x$ (columns and rows of $A$) whose indices belong to $\mathcal{B}$. Finally, the cardinality of $\mathcal{B}$ is denoted as $|\mathcal{B}|$.
\par The rest of the paper is organized as follows. In  Section \ref{section Algorithmic Framework}, we provide the algorithmic framework of the method. Consequently, in Section \ref{section Polynomial Convergence}, we prove polynomial complexity of the algorithm, and global convergence is established. In Section \ref{section Infeasible problems}, a necessary condition for infeasibility is derived, which is later used to construct an infeasibility detection mechanism. Numerical results of the implemented method are presented and discussed in Section \ref{section numerical results}. Finally, we derive some conclusions in Section \ref{section conclusions}.
\section{Algorithmic Framework} \label{section Algorithmic Framework}

\par In the previous section, we presented all the necessary tools for deriving the proposed Interior Point-Proximal Method of Multipliers (IP-PMM). Effectively, we would like to merge the proximal method of multipliers with the infeasible interior point method. For that purpose, assume that, at some iteration $k$ of the method, we have available an estimate $\lambda_k$ for a Lagrange multiplier vector.  Similarly, we denote by $\zeta_k$ the estimate of a primal solution. We define the proximal penalty function that has to be minimized at the $k$-th iteration of PMM, for solving (\ref{non-regularized primal}), given the estimates $\lambda_k,\ \zeta_k$:
\begin{equation} \label{PMM lagrangian}
\begin{split}
\mathcal{L}^{PMM}_{\mu_k} (x;\zeta_k, \lambda_k) = c^Tx + \frac{1}{2}x^T Q x -\lambda_k^T (Ax - b) + \frac{1}{2\mu_k}\|Ax-b\|_2^2 + \frac{\mu_k}{2}\|x-\zeta_k\|_2^2,
\end{split}
\end{equation}
\noindent with $\mu_k > 0$ some non-increasing penalty parameter. In order to solve the PMM sub-problem  (\ref{PMM sub-problem}), we will apply one (or a few) iterations of the previously presented infeasible IPM. To do that, we alter \eqref{PMM lagrangian}, by including logarithmic barriers, that is:
\begin{equation} \label{Proximal IPM Penalty}
\begin{split}
\mathcal{L}^{IP-PMM}_{\mu_k} (x;\zeta_k, \lambda_k) =\mathcal{L}^{PMM}_{\mu_k} (x;\zeta_k, \lambda_k) - \mu_k \sum_{j=1}^n \ln x^j,
\end{split}
\end{equation}
\noindent and we treat $\mu_k$ as the barrier parameter. In order to form the optimality conditions of this sub-problem, we equate the gradient of $\mathcal{L}^{IP-PMM}_{\mu_k}(\cdot;\lambda_k, \zeta_k)$ to the zero vector, i.e.:
\begin{equation*}
c + Qx - A^T \lambda_k + \frac{1}{\mu_k}A^T(Ax - b) + \mu_k (x - \zeta_k) - \mu_k X^{-1}e_n = 0.
\end{equation*}
\par Following the developments in \cite{paper_17}, we can define the variables: $y = \lambda_k - \frac{1}{\mu_k}(Ax - b)$ and $z = \mu_k X^{-1}e_n$, to get the following (equivalent) system of equations (first-order optimality conditions):
\begin{equation} \label{Proximal IPM FOC} \begin{bmatrix}
c + Qx - A^T y - z + \mu_k (x-\zeta_k)\\
Ax + \mu_k (y - \lambda_k) - b\\
Xz -  \mu_k e_n
\end{bmatrix} = \begin{bmatrix}
0\\
0\\
0
\end{bmatrix}.
\end{equation}

\par Let us introduce the following notation, that will be used later.
\begin{definition}
Let two real-valued positive functions be given: $T(x) : \mathbb{R}_+ \mapsto \mathbb{R}_+,\ f(x): \mathbb{R}_+ \mapsto \mathbb{R}_+$, with $\mathbb{R}_+ \coloneqq \{x \in \mathbb{R} : x \geq 0 \}$. We say that:
\begin{itemize}
\item[$\bullet$] $T(x) = O(f(x))$ (as $x \rightarrow \infty$) if and only if, there exist constants $c,\ x_0$, such that: $$ T(x) \leq c f(x),\ \ \text{for all}\ x \geq x_0.$$
\item[$\bullet$] $T(x) = \Omega(f(x))$ if and only if, there exist constants $c,\ x_0$, such that: $$T(x) \geq c f(x), \ \ \text{for all}\ x \geq x_0.$$
\item[$\bullet$] $T(x) = \Theta(f(x))$ if and only if, $T(x) = O(f(x))$ and $T(x) = \Omega(f(x)).$
\end{itemize}
\end{definition}

\noindent Next, given two arbitrary vectors $b \in \mathbb{R}^m,\ c \in \mathbb{R}^n$, we define the following semi-norm:
\begin{equation} \label{semi-norm definition}
\|(b,c)\|_{\mathcal{A}} \coloneqq \min_{x,y,z}\large\{\|(x,z)\|_2\ :\ Ax = b, -Qx + A^Ty + z = c\large\}.
\end{equation}
\noindent This semi-norm has been used before in \cite{paper_43}, as a way to measure infeasibility for the case of linear programming problems ($Q = 0$). For a discussion on the properties of the aforementioned semi-norm, as well as how one can evaluate it (using the QR factorization of $A$), the reader is referred to \cite[Section 4]{paper_43}.
\paragraph{Starting Point.}
\par  At this point, we are ready to define the starting point for IP-PMM. For that, we set $(x_0,z_0) = \rho(e_n,e_n)$, for some $\rho > 0$. We also set $y_0$ to some arbitrary vector (e.g. $y_0 = e_m$), such that $\|y_0\|_{\infty} = O(1)$ (i.e. the absolute value of its entries is independent of $n$ and $m$), and $\mu_0 = \frac{x_0^T z_0}{n}$. Then we have:
\begin{equation} \label{starting point}
Ax_0 = b + \bar{b},\ -Qx_0 + A^Ty_0 + z_0 = c + \bar{c},\  \zeta_0 = x_0,\ \lambda_0 = y_0.
\end{equation}
\noindent for some vectors $\bar{b},\ \bar{c}$.
\paragraph{Neighbourhood.}

\par We mentioned earlier that we develop a path-following method. Hence, we have to describe a neighbourhood in which the iterations of the method should lie. However, unlike typical path-following methods, we define a family of neighbourhoods that depend on the PMM sub-problem parameters. 
\par Given the starting point in \eqref{starting point}, penalty $\mu_k$, and estimates $\lambda_k, \zeta_k$, we define the following \textit{regularized set of centers}:
\begin{equation*}
\mathcal{P}_{\mu_k}(\zeta_k,\lambda_k) \coloneqq \large\{(x,y,z)\in \mathcal{C}_{\mu_k}(\zeta_k,\lambda_k)\ :\ (x,z) > (0,0),\ Xz = \mu_k e_n \large\},
\end{equation*}
\noindent where
\begin{equation*}
\mathcal{C}_{\mu_k}(\zeta_k,\lambda_k) \coloneqq \bigg\{(x,y,z)\ :\ \begin{matrix}
Ax + \mu_k(y-\lambda_k) = b + \frac{\mu_k}{\mu_0} \bar{b},\\
-Qx + A^Ty + z - \mu_k (x- \zeta_k) = c +  \frac{\mu_k}{\mu_0}\bar{c}
\end{matrix} \bigg\},
\end{equation*}
\noindent and $\bar{b},\ \bar{c}$ are as in \eqref{starting point}. The term set of centers originates from \cite{paper_43}, where a similar set is studied.
\par In order to enlarge the previous set, we define the following set:
\begin{equation*}
\tilde{\mathcal{C}}_{\mu_k}(\zeta_k,\lambda_k) \coloneqq \Bigg\{(x,y,z)\ :\ \begin{matrix}
Ax + \mu_k(y-\lambda_k) = b + \frac{\mu_k}{\mu_0} (\bar{b}+\tilde{b}_{k}),\\
-Qx + A^Ty + z - \mu_k (x- \zeta_k) = c +  \frac{\mu_k}{\mu_0}(\bar{c}+\tilde{c}_{k})\\
\|(\tilde{b}_{k},\tilde{c}_{k})\|_2 \leq C_N,\ \|(\tilde{b}_{k},\tilde{c}_{k})\|_{\mathcal{A}} \leq \gamma_{\mathcal{A}} \rho
\end{matrix} \Bigg\},
\end{equation*}
where $C_N > 0$ is a constant, $\gamma_{\mathcal{A}} \in (0,1)$, and $\rho>0$ is as defined in the starting point.  The vectors $\tilde{b}_{k}$ and $\tilde{c}_{k}$ represent the current scaled (by $\frac{\mu_0}{\mu_k}$) infeasibility and vary depending on the iteration $k$. In particular, these vectors can be formally defined recursively, depending on the iterations of IP-PMM. However, such a definition is not necessary for the developments to follow. In essence, the only requirement is that these scaled infeasibility vectors are bounded above by some constants, with respect to the 2-norm as well as the semi-norm defined in \eqref{semi-norm definition}. Using the latter set, we are now ready to define a family of neighbourhoods:
\begin{equation} \label{Small neighbourhood}
\mathcal{N}_{\mu_k}(\zeta_k,\lambda_k) \coloneqq \large\{(x,y,z) \in \tilde{\mathcal{C}}_{\mu_k}(\zeta_k,\lambda_k)\ :\ (x,z) >(0,0),\ x^i z^i \geq \gamma_{\mu} \mu_k,\ i \in \{1,\ldots,n\}\large\},
\end{equation}
\noindent where $\gamma_{\mu} \in (0,1)$ is a constant preventing component-wise complementarity products from approaching zero faster than $\mu_k = \frac{x_k^T z_k}{n}$. Obviously, the starting point defined in \eqref{starting point} belongs to the neighbourhood $\mathcal{N}_{\mu_0}(\zeta_0,\lambda_0)$, with $(\tilde{b}_{0},\tilde{c}_{0}) = (0,0)$.  Notice from the definition of the neighbourhood, that it depends on the choice of the constants $C_N$, $\gamma_{\mathcal{A}}$, $\gamma_{\mu}$. However, as the neighbourhood also depends on the parameters $\mu_k,\ \lambda_k,\ \zeta_k$, we omit the dependence on the constants, for simplicity of notation. 

\paragraph{Newton System.}
\par At every IP-PMM iteration, we approximately solve a perturbed form of the conditions in \eqref{Proximal IPM FOC}, by applying a variation of the Newton method. In particular, we form the Jacobian of the left-hand side of \eqref{Proximal IPM FOC} and we perturb the right-hand side of the Newton equation as follows: 
\begin{equation} \label{Newton System}
\begin{split}
	\begin{bmatrix} 
		-(Q + \mu_k I_n) &   A^T & I\\
		A  & \mu_k I_m & 0\\
		Z_k &  0  &X_k 
	\end{bmatrix}
	\begin{bmatrix}
		\Delta x_k\\
		\Delta y_k\\
		\Delta z_k
	\end{bmatrix}
	= \\ -
	\begin{bmatrix}
		- (c + \frac{\sigma_k \mu_k}{\mu_0}\bar{c}) - Qx_k  + A^T y_k + z_k -\sigma_k\mu_k (x_k - \zeta_k)\\
		Ax_k  +\sigma_k\mu_k (y_k - \lambda_k)- (b +\frac{\sigma_k \mu_k}{\mu_0}\bar{b}) \\
 		X_kZ_ke_n - \sigma_{k} \mu_k e_n
	\end{bmatrix},
	\end{split}
	\end{equation}
\noindent where $\bar{b},\ \bar{c}$ are as in \eqref{starting point}. Notice that we perturb the right-hand side of the Newton system in order to ensure that the iterates remain in the neighbourhood \eqref{Small neighbourhood}, while trying to reduce the value of the penalty (barrier) parameter $\mu_k$.
\par We are now able to derive Algorithm \ref{Algorithm PMM-IPM}, summarizing the proposed interior point-proximal method of multipliers. We will prove polynomial complexity of this scheme in the next section, under standard assumptions.
\renewcommand{\thealgorithm}{IP-PMM}

\begin{algorithm}[!ht]
\caption{Interior Point-Proximal Method of Multipliers}
    \label{Algorithm PMM-IPM}
    \textbf{Input:}  $A, Q, b, c$, $\text{tol}$.\\
    \textbf{Parameters:} $0< \sigma_{\min} \leq \sigma_{\max} \leq 0.5$, $C_N > 0$, $0<\gamma_{\mathcal{A}} < 1,\ 0<\gamma_{\mu} < 1$.\\
    \textbf{Starting point:} Set as in \eqref{starting point}.
\begin{algorithmic}[]
\For {($k= 0,1,2,\cdots$)}
\If {(($\|Ax_k - b\|_2 < \text{tol}) \wedge (\|c + Qx_k - A^T y_k - z_k\|_2 < \text{tol}) \wedge (\mu_k < \text{tol})$)}
	\State \Return $(x_k,y_k,z_k)$.	
\Else
	\State Choose $\sigma_k \in [\sigma_{\min},\sigma_{\max}]$ and solve \eqref{Newton System}.
	\State Choose step-length $\alpha_k$, as the largest $\alpha \in (0,1]$ such that  $\mu_k(\alpha)  \leq  \ (1-0.01 \alpha)\mu_k$ and:
	\begin{equation*}
	\begin{split}
	 (x_k + \alpha_k &\Delta x_k, y_k + \alpha_k \Delta y_k, z_k + \alpha_k \Delta z_k) \in  \ \mathcal{N}_{\mu_k(\alpha)}(\zeta_k,\lambda_k),\\ 
	\text{where, }\ \ & \mu_{k}(\alpha) = \frac{(x_k + \alpha_k \Delta x_k)^T(z_k + \alpha_k \Delta z_k)}{n}.
	\end{split}
	\end{equation*}	
	\State Set $(x_{k+1},y_{k+1},z_{k+1}) = (x_k + \alpha_k \Delta x_k, y_k + \alpha_k \Delta y_k, z_k + \alpha_k \Delta z_k)$, $\mu_{k+1} = \frac{x_{k+1}^Tz_{k+1}}{n}$.
	\State Let $r_p = Ax_{k+1} - (b + \frac{\mu_{k+1}}{\mu_0}\bar{b})$, $r_d = (c + \frac{\mu_{k+1}}{\mu_0}\bar{c}) + Qx_{k+1} - A^Ty_{k+1} -z_{k+1}.$
  	\If {\bigg(\big($\|(r_p,r_d)\|_2 \leq C_N \frac{\mu_{k+1}}{\mu_0} \big) \wedge \big(\|(r_p,r_d)\|_{\mathcal{A}} \leq \gamma_{\mathcal{A}}\rho \frac{\mu_{k+1}}{\mu_0}\big)$\bigg)}
  		\State $(\zeta_{k+1},\lambda_{k+1}) = (x_{k+1},y_{k+1})$.
  	\Else
  		\State $(\zeta_{k+1},\lambda_{k+1}) = (\zeta_{k},\lambda_{k})$.
  	\EndIf
\EndIf
\EndFor 
\end{algorithmic}
\end{algorithm}

\par Notice, in Algorithm \ref{Algorithm PMM-IPM}, that we force $\sigma$ to be less than $0.5$. This value is set, without loss of generality, for simplicity of exposition. Similarly, in the choice of the step-length, we require that $\mu_k(\alpha) \leq (1-0.01 \alpha)\mu_k$. The constant $0.01$ is chosen for ease of presentation. It depends on the choice of the maximum value of $\sigma$. The constants $C_N,\ \gamma_{\mathcal{A}},\ \gamma_{\mu}$, are used in the definition of the neighbourhood in \eqref{Small neighbourhood}. Their values can be considered to be arbitrary. The input $\text{tol}$, represents the error tolerance (chosen by the user). The terminating conditions require the Euclidean norm of primal and dual infeasibility, as well as complementarity, to be less than this tolerance. In such a case, we accept the iterate as a solution triple. The estimates $\lambda, \zeta$ are not updated if primal or dual feasibility are not both sufficiently decreased. In this case, we keep the estimates constant while continuing decreasing the penalty parameter $\mu_k$. Following the usual practice with proximal and augmented Lagrangian methods, we accept a new estimate when the respective residual is sufficiently decreased. However, the algorithm requires the evaluation of the semi-norm defined in \eqref{semi-norm definition}, at every iteration. While this is not practical, it can be achieved in polynomial time, with respect to the size of the problem. For a detailed discussion on this, the reader is referred to \cite[Section 4]{paper_43}.
 \section{Convergence Analysis of IP-PMM} \label{section Polynomial Convergence}
\par In this section we prove polynomial complexity and global convergence of Algorithm \ref{Algorithm PMM-IPM}. The proof follows by induction on the iterations of IP-PMM. That is, given an iterate $(x_k,y_k,z_k)$ at an arbitrary iteration $k$ of IP-PMM, we prove that the next iterate belongs to the appropriate neighbourhood required by the algorithm. In turn, this allows us to prove global and polynomial convergence of IP-PMM. An outline of the proof can be briefly explained as follows:
\begin{itemize}
\item[$\bullet$] Initially, we present some technical results in Lemmas \ref{Lemma non-expansiveness}-\ref{Lemma tilde point} which are required for the analysis throughout this section.
\item[$\bullet$] In turn, we prove boundedness of the iterates $(x_k,y_k,z_k)$ of IP-PMM in Lemma \ref{Lemma boundedness of x z}. In particular we show that $\|(x_k,y_k,z_k)\|_2 = O(n)$ and $\|(x_k,z_k)\|_1 = O(n)$ for every $k \geq 0$.
\item[$\bullet$] Then, we prove boundedness of the Newton direction computed at every IP-PMM iteration in Lemma \ref{Lemma boundedness Dx Dz}. More specifically, we prove that $\|(\Delta x_k,\Delta y_k,\Delta z_k)\|_2 = O(n^3)$ for every $k \geq 0$. 
\item[$\bullet$] In Lemma \ref{Lemma step-length} we prove the existence of a positive step-length $\bar{\alpha}$ so that the new iterate of IP-PMM, $(x_{k+1},y_{k+1},z_{k+1})$, belongs to the updated neighbourhood $\mathcal{N}_{\mu_{k+1}}(\zeta_{k+1},\lambda_{k+1})$, for every $k \geq 0$. In particular, we show that $\bar{\alpha} \geq \frac{\bar{\kappa}}{n^4}$, where $\bar{\kappa}$ is a constant independent of $n$ and $m$.
\item[$\bullet$] Q-linear convergence of the barrier parameter $\mu_k$ to zero is established in Theorem \ref{Theorem mu convergence}.
\item[$\bullet$] The polynomial complexity of IP-PMM is then proved in Theorem \ref{Theorem complexity}, showing that IP-PMM converges to an $\epsilon$-accurate solution in at most $O(n^4|\log\big(\frac{1}{\epsilon}\big)|)$ steps.
\item[$\bullet$] Finally, global converge to an optimal solution of \eqref{non-regularized primal}-\eqref{non-regularized dual} is established in Theorem \ref{Theorem convergence for the feasible case}. 
\end{itemize} 

\par For the rest of this section, we will make use of the following two assumptions, which are commonly employed when analyzing the complexity of an IPM. 
\begin{assumption} \label{Assumption 1}
There exists an optimal solution $(x^*,y^*,z^*)$ for the primal-dual pair \textnormal{(\ref{non-regularized primal})-(\ref{non-regularized dual})}, such that $\|x^*\|_{\infty} \leq C^*$, $\|y^*\|_{\infty} \leq C^*$ and $\|z^*\|_{\infty} \leq C^*$, for some constant $C^* \geq 0$, independent of $n$ and $m$.
\end{assumption}

\begin{assumption} \label{Assumption 2}
The constraint matrix of \textnormal{\eqref{non-regularized primal}} has full row rank, that is $\text{rank}(A) = m$. Furthermore, we assume that there exist constants $C_A^1 > 0$, $C_A^2 > 0$, $C_Q >0$ and $C_r > 0$, independent of $n$ and $m$, such that:
\begin{equation*}
\eta_{\min}(A) \geq C_A^1,\quad \eta_{\max}(A) \leq C_A^2,\quad \nu_{\max}(Q) \leq C_Q,\quad \|(c,b)\|_{\infty}\leq C_r.
\end{equation*}
\end{assumption}
\noindent Note that the independence of the previous constants from the problem dimensions is assumed for simplicity of exposition; this is a common practice when analyzing the complexity of interior point methods. If these constants depend polynomially on $n$ (or $m$), the analysis still holds by suitably altering the worst-case polynomial bound for the number of iterations of the algorithm. 
\par Let us now use the properties of the proximal operator defined in \eqref{Primal Dual Proximal Operator}. 

\begin{lemma} \label{Lemma non-expansiveness}
Given Assumption \textnormal{\ref{Assumption 1}}, and for all $\lambda \in \mathbb{R}^m$, $\zeta \in \mathbb{R}^n$ and $0 \leq \mu < \infty$, there exists  a unique pair $(x_r^*,y_r^*)$, such that $(x_r^*,y_r^*) = P(\zeta,\lambda),$ and
\begin{equation} \label{non-expansiveness property}
\|(x_r^*,y_r^*)-(x^*,y^*)\|_2 \leq \|(\zeta,\lambda)-(x^*,y^*)\|_2,
\end{equation}
\noindent where $P(\cdot)$ is defined as in \eqref{Primal Dual Proximal Operator} and $(x^*,y^*)$ is the same as in Assumption \textnormal{\ref{Assumption 1}}.
\end{lemma}
\begin{proof}
\par We know that $P(\cdot,\cdot)$ is single-valued and non-expansive (\cite{paper_39}), and hence there exists a unique pair $(x_r^*,y_r^*)$, such that $(x_r^*,y_r^*) = P(\zeta,\lambda),$ for all $\lambda,\ \zeta$ and $0 \leq \mu < \infty$. Given the optimal triple of Assumption \ref{Assumption 1}, we can use the non-expansiveness of $P(\cdot)$ in \eqref{Primal Dual Proximal Operator}, to show that:
\begin{equation*} 
\|P_k(\zeta,\lambda) - P_k(x^*,y^*)\|_2 = \|(x_r^*,y_r^*)-(x^*,y^*)\|_2 \leq \|(\zeta,\lambda)-(x^*,y^*)\|_2,
\end{equation*} 
\noindent where we used the fact that $P(x^*,y^*) = (x^*,y^*)$, which follows directly from \eqref{Primal Dual Maximal Monotone Operator}, as we can see that $(0,0) \in T_{\mathcal{L}}(x^*,y^*)$. This completes the proof. 
\end{proof}
\begin{lemma} \label{Lemma-boundedness of optimal solutions for sub-problems}
Given Assumptions \textnormal{\ref{Assumption 1}, \ref{Assumption 2}}, there exists a triple $(x_{r_k}^*,y_{r_k}^*,z_{r_k}^*)$, satisfying: 
\begin{equation} \label{PMM optimal solution}
\begin{split}
A x_{r_k}^* + \mu (y_{r_k}^*-\lambda_k) -b & =   0\\
-c-Qx_{r_k}^* + A^T y_{r_k}^* + z_{r_k}^* - \mu (x_{r_k}^* - \zeta_k)& =  0,\\
(x_{r_k}^*)^T(z_{r_k}^*) & = 0,
\end{split}
\end{equation}
\noindent with $\|(x_{r_k}^*,y_{r_k}^*,z_{r_k}^*)\|_2 = O(\sqrt{n})$, for all $\lambda_k \in \mathbb{R}^m$, $\zeta_k \in \mathbb{R}^n$ produced by Algorithm \textnormal{\ref{Algorithm PMM-IPM}}, and any $\mu \in [0,\infty)$. Moreover, we have that $\|(\zeta_k,\lambda_k)\|_2 = O(\sqrt{n})$, for all $k \geq 0$.
\end{lemma}
\begin{proof}
\par We prove the claim by induction on the iterates, $k \geq 0$, of Algorithm \ref{Algorithm PMM-IPM}. At iteration $k = 0$, we have that $\lambda_0 = y_0$ and $\zeta_0 = x_0$. But from the construction of the starting point in \eqref{starting point}, we know that $\|(x_0,y_0)\|_2 = O(\sqrt{n})$. Hence, $\|(\zeta_0,\lambda_0)\|_2 = O(\sqrt{n})$ (assuming that $n > m$). From Lemma \ref{Lemma non-expansiveness}, we know that there exists a unique pair $(x_{r_0}^*,y_{r_0}^*)$ such that:
$$(x_{r_0}^*,y_{r_0}^*) = P_0(\zeta_0,\lambda_0),\quad \textnormal{and} \quad \|(x_{r_0}^*,y_{r_0}^*) - (x^*,y^*)\|_2 \leq \|(\zeta_0,\lambda_0)-(x^*,y^*)\|_2.$$
\noindent Using the triangular inequality, and combining the latter inequality with our previous observations, as well as Assumption \ref{Assumption 1}, yields that $\|(x_{r_0}^*,y_{r_0}^*)\|_2 = O(\sqrt{n})$. From the definition of the operator in \eqref{PMM Operator Subproblem}, we know that: 
\begin{equation*}
\begin{split}
 -c - Qx_{r_0}^* + A^T y_{r_0}^* - \mu (x_{r_0}^* - \zeta_k) &\ \in \partial \delta_+(x_{r_0}^*),\\
 Ax_{r_0}^* + \mu (y_{r_0}^*-\lambda_k) - b &\ = 0,
\end{split}
\end{equation*}
\noindent where $\partial(\delta_+(\cdot))$ is the sub-differential of the indicator function defined in \eqref{Indicator function}. Hence, we know that there must exist $-z_{r_0}^* \in \partial \delta_+(x_{r_0}^*)$ (and hence, $z_{r_0}^* \geq 0$, $(x_{r_0}^*)^T(z_{r_0}^*) = 0$), such that:
$$z_{r_0}^* = c + Qx_{r_0}^* - A^T y_{r_0}^* + \mu (x_{r_0}^* - \zeta_k),\ (x^*_{r_0})^T(z^*_{r_0}) = 0,\ \|z_{r_0}^*\|_2 = O(\sqrt{n}),$$
\noindent where $\|z_{r_0}^*\|_2 = O(\sqrt{n})$ follows from Assumption \ref{Assumption 2}, combined with $\|(x_0,y_0)\|_2 = O(\sqrt{n})$.

\par Let us now assume that at some iteration $k$ of Algorithm \ref{Algorithm PMM-IPM}, we have $\|(\zeta_k,\lambda_k)\|_2 = O(\sqrt{n})$. There are two cases for the subsequent iteration:
\begin{itemize}
\item[\textbf{1.}] The proximal estimates are updated, that is $(\zeta_{k+1},\lambda_{k+1}) = (x_{k+1},y_{k+1})$, or
\item[\textbf{2.}] the proximal estimates stay the same, i.e. $(\zeta_{k+1},\lambda_{k+1}) = (\zeta_k,\lambda_k)$.
\end{itemize}
\par \textbf{Case 1.} We know by construction that this case can only occur if the following condition is satisfied:
$$\|(r_p,r_d)\|_2 \leq C_N \frac{\mu_{k+1}}{\mu_0},$$
\noindent where $r_p,\ r_d$ are defined in Algorithm \ref{Algorithm PMM-IPM}. However, from the neighbourhood conditions in \eqref{Small neighbourhood}, we know that:
$$\|\big(r_p + \mu_{k+1}(y_{k+1}-\lambda_k), r_d + \mu_{k+1}(x_{k+1}-\zeta_k)\big)\|_2 \leq C_N \frac{\mu_{k+1}}{\mu_0}.$$
\noindent Combining the last two inequalities by applying the triangular inequality, and using the inductive hypothesis ($\|(\lambda_k,\zeta_k)\|_2 = O(\sqrt{n})$), yields that $\|(x_{k+1},y_{k+1})\|_2 = O(\sqrt{n})$. Hence, $\|(\zeta_{k+1},\lambda_{k+1})\|_2 = O(\sqrt{n})$. Then, we can invoke Lemma \ref{Lemma non-expansiveness}, with $\lambda = \lambda_{k+1}$, $\zeta = \zeta_{k+1}$ and $\mu = \mu_{k+1}$, which gives that:
$$\|(x_{r_{k+1}}^*,y_{r_{k+1}}^*) - (x^*,y^*)\|_2 \leq \|(\zeta_{k+1},\lambda_{k+1})-(x^*,y^*)\|_2.$$
\noindent A simple manipulation yields that $\|(x_{r_{k+1}}^*,y_{r_{k+1}}^*)\|_2 = O(\sqrt{n})$. As before, we use \eqref{PMM Operator Subproblem} alongside Assumption \ref{Assumption 2} to show the existence of $-z_{r_{k+1}}^* \in \partial \delta_+(x_{r_{k+1}}^*)$, such that  the triple $(x_{r_{k+1}}^*,y_{r_{k+1}}^*,z_{r_{k+1}}^*)$ satisfies \eqref{PMM optimal solution} with $\|z_{r_{k+1}}^*\|_2 = O(\sqrt{n})$.

\par \textbf{Case 2.} In this case, we have $(\zeta_{k+1},\lambda_{k+1}) = (\zeta_k,\lambda_k)$, and hence the inductive hypothesis gives us directly that: $\|(\zeta_{k+1},\lambda_{k+1})\|_2 = O(\sqrt{n})$. The same reasoning as before implies the existence of a triple $(x_{r_{k+1}}^*,y_{r_{k+1}}^*,z_{r_{k+1}}^*)$ satisfying \eqref{PMM optimal solution}, with $\|(x_{r_{k+1}}^*,y_{r_{k+1}}^*,z_{r_{k+1}}^*)\|_2 = O(\sqrt{n})$. 
\end{proof}

\begin{lemma} \label{Lemma tilde point}
Given Assumptions \textnormal{\ref{Assumption 1}, \ref{Assumption 2}}, $\lambda_k$ and $\zeta_k$ produced at an arbitrary iteration $k \geq 0$ of Algorithm \textnormal{\ref{Algorithm PMM-IPM}} and any $\mu \in [0,\infty)$, there exists a triple $(\tilde{x},\tilde{y},\tilde{z})$ which satisfies the following system of equations:
\begin{equation}  \label{tilde point conditions}
\begin{split}
A \tilde{x} + \mu \tilde{y} & =   b + \bar{b} + \mu \lambda_k + \tilde{b}_k,\\
-Q\tilde{x} + A^T \tilde{y} + \tilde{z} - \mu \tilde{x} & =  c + \bar{c} - \mu \zeta_k + \tilde{c}_k,\\
\tilde{X}\tilde{z} & = \theta e_n,
\end{split}
\end{equation}
\noindent for some arbitrary $\theta > 0$, with $(\tilde{x},\tilde{z}) > (0,0)$ and $\|(\tilde{x},\tilde{y},\tilde{z})\|_2 = O(\sqrt{n})$, where $\tilde{b}_{k},\ \tilde{c}_{k}$ are defined in \eqref{Small neighbourhood}, while $\bar{b},\ \bar{c}$ are defined with the starting point in \eqref{starting point}.
\end{lemma}
\begin{proof}
\par Let $k \geq 0$ denote an arbitrary iteration of Algorithm \ref{Algorithm PMM-IPM}. Let also $\bar{b},\ \bar{c}$ as defined in \eqref{starting point}, and $\tilde{b}_{k},\ \tilde{c}_{k}$, as defined in the neighbourhood conditions in \eqref{Small neighbourhood}. Given an arbitrary positive constant $\theta > 0$, we consider the following barrier primal-dual pair:
\begin{equation} \label{tilde non-regularized primal} 
\text{min}_{x} \ \big( (c+\bar{c} + \tilde{c}_k)^Tx + \frac{1}{2}x^T Q x -\theta \sum_{j=1}^n \ln x^j \big), \ \ \text{s.t.}  \  Ax = b + \bar{b} + \tilde{b}_k,   
\end{equation}
\begin{equation} \label{tilde non-regularized dual} 
\text{max}_{x,y,z}  \ \big((b + \bar{b} + \tilde{b}_k)^Ty - \frac{1}{2}x^T Q x+\theta \sum_{j=1}^n \ln z^j\big), \ \ \text{s.t.}\  -Qx + A^Ty + z = c+\bar{c} + \tilde{c}_k.
\end{equation}
\noindent Let us now define the following triple:
\begin{equation*}
(\hat{x},\hat{y},\hat{z}) \coloneqq \arg \min_{(x,y,z)} \big\{\|(x,z)\|_2: Ax = \tilde{b}_k,\ -Qx + A^T y + z = \tilde{c}_k \}. 
\end{equation*}
\noindent From the neighbourhood conditions \eqref{Small neighbourhood}, we know that $\|(\tilde{b}_k,\tilde{c}_k)\|_{\mathcal{A}} \leq \gamma_{\mathcal{A}}\rho$, and from the definition of the semi-norm in \eqref{semi-norm definition}, we have that: $\|(\hat{x},\hat{z})\|_2 \leq \gamma_{\mathcal{A}} \rho$. Using \eqref{semi-norm definition} alongside Assumption \ref{Assumption 2}, we can also show that $\|\hat{y}\|_2 = \Theta(\|(\hat{x},\hat{z})\|_2)$. On the other hand, from the definition of the starting point, we have that: $(x_0,z_0) = \rho(e_n,e_n)$. By defining the following auxiliary point:
$$(\bar{x},\bar{y},\bar{z}) = (x_0,y_0,z_0) + (\hat{x},\hat{y},\hat{z}),$$
\noindent we have that $(\bar{x},\bar{z}) \geq (1-\gamma_{\mathcal{A}})\rho(e_n,e_n)$. By construction, the triple $(\bar{x},\bar{y},\bar{z})$ is a feasible solution for the primal-dual pair in \eqref{tilde non-regularized primal}-\eqref{tilde non-regularized dual}, giving bounded primal and dual objective values, respectively.
\par Using our previous observations, alongside the fact that $\text{rank}(A) = m$ (Assumption \ref{Assumption 2}), we can confirm that there must exist a large constant $M > 0$, independent of $n$, and a triple $(x_s^*,y_s^*,z_s^*)$ solving \eqref{tilde non-regularized primal}-\eqref{tilde non-regularized dual}, such that $\|(x_s^*,y_s^*,z_s^*)\|_{\infty} \leq M  \Rightarrow \|(x_s^*,y_s^*,z_s^*)\|_2 = O(\sqrt{n}).$ 
%Slater's constraint qualification holds for both \eqref{tilde non-regularized primal} and \eqref{tilde non-regularized dual}. By invoking Nonlinear Farka's Lemma (see \cite[Prop. 6.4.1]{book_2}) in both \eqref{tilde non-regularized primal} and \eqref{tilde non-regularized dual},
\par Let us now apply the proximal method of multipliers to \eqref{tilde non-regularized primal}-\eqref{tilde non-regularized dual}, given the estimates $\zeta_k,\ \lambda_k$. We should note at this point, that the proximal operator used here is different from that in \eqref{Primal Dual Proximal Operator}, since it is based on a different maximal monotone operator from that in \eqref{Primal Dual Maximal Monotone Operator}. In particular, we associate the following maximal monotone operator to \eqref{tilde non-regularized primal}-\eqref{tilde non-regularized dual}:
\begin{equation*}
\tilde{T}_{\mathcal{L}}(x,y) = \{(u,v): v = Qx + (c + \bar{c} + \tilde{c}_k) - A^Ty - \theta X^{-1}e_n,\ u = Ax-(b+\bar{b}+\tilde{b}_k) \},
\end{equation*}
\noindent As before, the proximal operator is defined as: $\tilde{P} = (I_{m+n} + \tilde{T}_{\mathcal{L}})^{-1}$, and is single-valued and non-expansive.  We let any $\mu \in [0,\infty)$ and define the following penalty function:
\begin{equation*} 
\begin{split}
\tilde{\mathcal{L}}_{\mu}(x;\zeta_k,\lambda_k) = (c + \bar{c} + \tilde{c}_k)^T x + \frac{1}{2}x^TQx\ + \\
 \frac{1}{2}\mu \|x-\zeta_k\|_2^2 + \frac{1}{2\mu}\|Ax-(b+\bar{b}+\tilde{b}_k)\|_2^2 - (\lambda_k)^T(Ax - (b+\bar{b}+\tilde{b}_k))-\theta \sum_{j=1}^n \ln x^j.
\end{split}
\end{equation*}
\noindent By defining the variables: $y = \lambda_k - \frac{1}{\mu}(Ax - (b+\bar{b}+\tilde{b}_k))$ and $z = \theta X^{-1}e_n$, we can see that the optimality conditions of this PMM sub-problem are exactly those stated in \eqref{tilde point conditions}. Equivalently, we can find a pair $(\tilde{x},\tilde{y})$ such that $(\tilde{x},\tilde{y}) = \tilde{P}(\zeta_k,\lambda_k)$ and set $\tilde{z} = \theta \tilde{X}^{-1}e_n$. In order to conclude the proof, we can use non-expansiveness of $\tilde{P}$, as in Lemma \ref{Lemma non-expansiveness}, to get that:
$$\|(\tilde{x},\tilde{y})-(x_s^*,y_s^*)\|_2 \leq \|(\zeta_k,\lambda_k)-(x_s^*,y_s^*)\|_2.$$
\noindent But we know, from Lemma \ref{Lemma-boundedness of optimal solutions for sub-problems}, that $\|(\zeta_k,\lambda_k)\|_2 = O(\sqrt{n})$, $\forall\ k \geq 0$. Combining this with our previous observations yields that $\|(\tilde{x},\tilde{y})\|_2 = O(\sqrt{n})$. Setting $\tilde{z} = \theta\tilde{X}^{-1}e_n$, gives a triple $(\tilde{x},\tilde{y},\tilde{z})$ that satisfies \eqref{tilde point conditions}, while $\|(\tilde{x},\tilde{y},\tilde{z})\|_2 = O(\sqrt{n})$. This concludes the proof. 
\end{proof}

\noindent In the following Lemma, we derive boundedness of the iterates of Algorithm \ref{Algorithm PMM-IPM}.

\begin{lemma} \label{Lemma boundedness of x z}
Given Assumptions \textnormal{\ref{Assumption 1}} and \textnormal{\ref{Assumption 2}}, the iterates $(x_k,y_k,z_k)$ produced by Algorithm \textnormal{\ref{Algorithm PMM-IPM}}, for all $k \geq 0$, are such that:
$$\|(x_k,z_k)\|_1 = O(n),\ \|(x_k,y_k,z_k)\|_2 = O(n).$$
\end{lemma}
\begin{proof}
\par Let an iterate $(x_k,y_k,z_k) \in \mathcal{N}_{\mu_k}(\zeta_k,\lambda_k)$, produced by Algorithm \ref{Algorithm PMM-IPM} during an arbitrary iteration $k \geq 0$, be given. Firstly, we invoke Lemma \ref{Lemma tilde point}, from which we have a triple $(\tilde{x},\tilde{y},\tilde{z})$ satisfying \eqref{tilde point conditions}, for $\mu = \mu_k$. Similarly, by invoking Lemma \ref{Lemma-boundedness of optimal solutions for sub-problems}, we know that there exists a triple $(x_{r_k}^*,y_{r_k}^*,z_{r_k}^*)$ satisfying \eqref{PMM optimal solution}, with $\mu = \mu_k$. Consider the following auxiliary point:
\begin{equation} \label{auxiliary triple 1}
\bigg((1-\frac{\mu_k}{\mu_0})x_{r_k}^* + \frac{\mu_k}{\mu_0} \tilde{x} - x_k,\ (1-\frac{\mu_k}{\mu_0})y_{r_k}^* +\frac{\mu_k}{\mu_0} \tilde{y} - y_k,\ (1-\frac{\mu_k}{\mu_0})z_{r_k}^* + \frac{\mu_k}{\mu_0} \tilde{z} - z_k\bigg).
\end{equation}
\noindent  Using (\ref{auxiliary triple 1}) and \eqref{PMM optimal solution}-\eqref{tilde point conditions} (for $\mu = \mu_k$), one can observe that:
\begin{equation*}
\begin{split}
A((1-\frac{\mu_k}{\mu_0})x_{r_k}^* + \frac{\mu_k}{\mu_0} \tilde{x} - x_k) + \mu_k ((1-\frac{\mu_k}{\mu_0})y_{r_k}^* + \frac{\mu_k}{\mu_0} \tilde{y} - y_k) = \\
(1-\frac{\mu_k}{\mu_0})(Ax_{r_k}^* + \mu_k y_{r_k}^*) + \frac{\mu_k}{\mu_0} (A\tilde{x}+ \mu_k \tilde{y}) - Ax_k -\mu_k y_k =\\
(1-\frac{\mu_k}{\mu_0}) (b + \mu_k \lambda_k) + \frac{\mu_k}{\mu_0} (b + \mu_k \lambda_k + \tilde{b} +\bar{b}) - Ax_k - \mu_k y_k =\\
b +\mu_k \lambda_k + \frac{\mu_k}{\mu_0}(\tilde{b}+\bar{b}) - Ax_k - \mu_k y_k = &\ 0,
\end{split}
\end{equation*}
\noindent where the last equality follows from the definition of the neighbourhood $\mathcal{N}_{\mu_k}(\zeta_k,\lambda_k)$. Similarly:
\begin{equation*}
-(Q+\mu_k I_n)((1-\frac{\mu_k}{\mu_0})x_{r_k}^* + \frac{\mu_k}{\mu_0} \tilde{x} - x_k) + A^T((1-\frac{\mu_k}{\mu_0})y_{r_k}^* +\frac{\mu_k}{\mu_0} \tilde{y} - y_k) + ((1-\frac{\mu_k}{\mu_0})z_{r_k}^* + \frac{\mu_k}{\mu_0} \tilde{z} - z_k) = 0.
\end{equation*}
\noindent By combining the previous two relations, we have:
\begin{equation} \label{Lemma boundedness of x,z, relation 1}
\begin{split}
((1-\frac{\mu_k}{\mu_0})x_{r_k}^* + \frac{\mu_k}{\mu_0} \tilde{x} - x_k)^T((1-\frac{\mu_k}{\mu_0})z_{r_k}^* + \frac{\mu_k}{\mu_0} \tilde{z} - z_k) = &\\
((1-\frac{\mu_k}{\mu_0})x_{r_k}^* + \frac{\mu_k}{\mu_0}\tilde{x} - x_k)^T(Q+\mu_k I) ((1-\frac{\mu_k}{\mu_0})x_{r_k}^* + \frac{\mu_k}{\mu_0} \tilde{x} - x_k)\  +\\ \mu_k ((1-\frac{\mu_k}{\mu_0})y_{r_k}^* + \frac{\mu_k}{\mu_0} \tilde{y} - y_k)^T((1-\frac{\mu_k}{\mu_0})y_{r_k}^* + \frac{\mu_k}{\mu_0} \tilde{y} - y_k) \geq &\ 0.
\end{split}
\end{equation}
\noindent From (\ref{Lemma boundedness of x,z, relation 1}), it can be seen that:
\begin{equation*}
\begin{split}
((1-\frac{\mu_k}{\mu_0})x_{r_k}^* + \frac{\mu_k}{\mu_0} \tilde{x})^T z_k + ((1-\frac{\mu_k}{\mu_0})z_{r_k}^* + \frac{\mu_k}{\mu_0}\tilde{z})^T x_k \leq \\
((1-\frac{\mu_k}{\mu_0})x_{r_k}^* + \frac{\mu_k}{\mu_0} \tilde{x})^T((1-\frac{\mu_k}{\mu_0})z_{r_k}^* + \frac{\mu_k}{\mu_0} \tilde{z}) + x_k^T z_k.
\end{split}
\end{equation*}
\noindent However, from Lemmas \ref{Lemma-boundedness of optimal solutions for sub-problems} and \ref{Lemma tilde point}, we have that: $(\tilde{x},\tilde{z})  \geq \xi (e_n,e_n)$, for some positive constant $\xi = \Theta(1)$, while $\|(x_{r_k}^*,z_{r_k}^*)\|_2 = O(\sqrt{n})$, and, $\|(\tilde{x},\tilde{z})\|_2 = O(\sqrt{n})$. Furthermore, by definition we have that $ n \mu_k = x_k^Tz_k$. By combining all the previous, we obtain:
\begin{equation} \label{Lemma boundedness of x,z, relation 2}
\begin{split}
\frac{\mu_k}{\mu_0} \xi(e^T x_k + e^T z_k) \leq \\
((1-\frac{\mu_k}{\mu_0})x_{r_k}^* + \frac{\mu_k}{\mu_0} \tilde{x})^T z_k + ((1-\frac{\mu_k}{\mu_0})z_{r_k}^* + \frac{\mu_k}{\mu_0} \tilde{z})^T x_k \leq \\
((1-\frac{\mu_k}{\mu_0})x_{r_k}^* + \frac{\mu_k}{\mu_0} \tilde{x})^T((1-\frac{\mu_k}{\mu_0})z_{r_k}^* + \frac{\mu_k}{\mu_0} \tilde{z}) + x_k^T z_k = \\
\frac{\mu_k}{\mu_0}(1-\frac{\mu_k}{\mu_0}) (x_{r_k}^*)^T \tilde{z} + \frac{\mu_k}{\mu_0} (1-\frac{\mu_k}{\mu_0}) \tilde{x}^T z_{r_k}^* + (\frac{\mu_k}{\mu_0})^2 \tilde{x}^T \tilde{z} + x_k^T z_k = &\ O(\mu_k n),
\end{split}
\end{equation}
\noindent where we used \eqref{PMM optimal solution} ($(x_{r_k}^*)^T(z_{r_k}^*) = 0$). Hence, (\ref{Lemma boundedness of x,z, relation 2}) implies that:
$$\|(x_k,z_k)\|_1 = O(n).$$
\noindent From equivalence of norms, we have that $\|(x_k,z_k)\|_2 \leq \|(x_k,z_k)\|_1$. Finally, from the neighbourhood conditions we know that:
$$c+ Qx_k - A^T y_k - z_k + \mu_k(x_k - \zeta_k) + \frac{\mu_k}{\mu_0} (\tilde{c}_k + \bar{c}) = 0.$$
\noindent All terms above (except for $y_k$) have a 2-norm that is $O(n)$ (note that $\|(\bar{c},\bar{b})\|_2 = O(\sqrt{n})$ using Assumption \ref{Assumption 2} and the definition in \eqref{starting point}). Hence, using again Assumption \ref{Assumption 2} yields that $\|y_k\|_2 = O(n)$, and completes the proof. 
\end{proof}

\noindent As in a typical IPM convergence analysis, we proceed by bounding some components of the scaled Newton direction. The proof of that uses similar arguments to those presented in \cite[Lemma 6.5]{book_4}. Combining this result with Assumption \ref{Assumption 2}, allows us to bound also the unscaled Newton direction. 
\begin{lemma} \label{Lemma boundedness Dx Dz}
Given Assumptions \textnormal{\ref{Assumption 1}} and \textnormal{\ref{Assumption 2}}, and the Newton direction $(\Delta x_k, \Delta y_k, \Delta z_k)$ obtained by solving system \textnormal{(\ref{Newton System})} during an arbitrary iteration $k \geq 0$ of Algorithm \textnormal{\ref{Algorithm PMM-IPM}}, we have that:
$$\|D_k^{-1}\Delta x_k\|_2 =  O(n^{2}\mu^{\frac{1}{2}}),\ \|D_k \Delta z_k \|_2 = O(n^{2}\mu^{\frac{1}{2}}),\ \|(\Delta x_k,\Delta y_k,\Delta z_k)\|_2 = O(n^{3}),$$
with $D_k^2 = X_k Z_k^{-1}$.
\end{lemma}

\begin{proof}
\par Consider an arbitrary iteration $k$ of Algorithm \ref{Algorithm PMM-IPM}. We invoke Lemmas \ref{Lemma-boundedness of optimal solutions for sub-problems}, \ref{Lemma tilde point}, for $\mu = \sigma_k \mu_k$. That is, there exists a triple $(x_{r_k}^*,y_{r_k}^*,z_{r_k}^*)$ satisfying \eqref{PMM optimal solution}, and a triple $(\tilde{x},\tilde{y},\tilde{z})$ satisfying \eqref{tilde point conditions}, for $\mu = \sigma_k \mu_k$. Using the centering parameter $\sigma_k$, we define the following vectors: 
\begin{equation} \label{Boundedness dx dz c hat b hat equation}
 \hat{c} = -\bigg(\sigma_k \frac{\bar{c}}{\mu_0} - (1-\sigma_k)\big(x_k - \zeta_k + \frac{\mu_k}{\mu_0}(\tilde{x}-x_{r_k}^*)\big)\bigg),\ \hat{b} =  -\bigg(\sigma_k \frac{\bar{b}}{\mu_0} + (1-\sigma_k)\big(y_k - \lambda_k +\frac{\mu_k}{\mu_0}(\tilde{y}-y_{r_k}^*)\big)\bigg),
\end{equation}
\noindent where $\bar{b},\ \bar{c},\ \mu_0$ are defined in \eqref{starting point}. Using Lemmas \ref{Lemma-boundedness of optimal solutions for sub-problems}, \ref{Lemma tilde point}, \ref{Lemma boundedness of x z}, and Assumption \ref{Assumption 2}, we know that $\|(\hat{c},\hat{b})\|_2 = O(n)$. Then, by applying again Assumption \ref{Assumption 2}, we know that there must exist a vector $\hat{x}$ such that: $A\hat{x} = \hat{b},\ \|\hat{x}\|_2 = O(n)$, and by setting $\hat{z} =  \hat{c} + Q\hat{x} + \mu \hat{x}$, we have that $\|\hat{z}\|_2 = O(n)$ and:
\begin{equation} \label{Boundedness of Dx,Dz, hat point}
 A\hat{x} = \hat{b},\ -Q\hat{x}+ \hat{z} - \mu_k \hat{x} = \hat{c}.
 \end{equation}
\par Using $(x_{r_k}^*,y_{r_k}^*,z_{r_k}^*)$, $(\tilde{x},\tilde{y},\tilde{z})$, as well as the triple $(\hat{x},0,\hat{z})$, where $(\hat{x},\hat{z})$ is defined in \eqref{Boundedness of Dx,Dz, hat point}, we can define the following auxiliary triple:
\begin{equation} \label{Lemma Dx Dz boundedness, auxiliary triple}
(\bar{x},\bar{y},\bar{z}) = (\Delta x_k, \Delta y_k, \Delta z_k) + \frac{\mu_k}{\mu_0} (\tilde{x}, \tilde{y}, \tilde{z}) - \frac{\mu_k}{\mu_0} (x_{r_k}^*, y_{r_k}^*,z_{r_k}^*) + \mu_k (\hat{x},0,\hat{z}).
\end{equation}
\noindent Using \eqref{Lemma Dx Dz boundedness, auxiliary triple}, \eqref{Boundedness dx dz c hat b hat equation}, \eqref{PMM optimal solution}-\eqref{tilde point conditions} (with $\mu = \sigma_k \mu_k$), and the second block equation of \eqref{Newton System}, we have:
\begin{equation*}
\begin{split}
A\bar{x} + \mu_k \bar{y} = &\ (A \Delta x_k + \mu_k \Delta y_k) + \frac{\mu_k}{\mu_0}((A\tilde{x} +  \mu_k \tilde{y})- (Ax_{r_k}^*+  \mu_k y_{r_k}^*)) + \mu_k A\hat{x}\\
= &\ \big(b + \sigma_k\frac{\mu_k}{\mu_0}\bar{b}-Ax_k - \sigma_k \mu_k(y_k-\lambda_k)\big) + \frac{\mu_k}{\mu_0}((A\tilde{x} +  \mu_k \tilde{y})- (Ax_{r_k}^*+  \mu_k y_{r_k}^*))\\
 &\ - \mu_k \big(\sigma_k \frac{\bar{b}}{\mu_0} + (1-\sigma_k)(y_k-\lambda_k)\big) - \frac{\mu_k}{\mu_0}(1-\sigma_k)\mu_k(\tilde{y}-y_{r_k}^*)\\
= &\ \big(b + \sigma_k\frac{\mu_k}{\mu_0}\bar{b}-Ax_k - \sigma_k \mu_k(y_k-\lambda_k)\big) + \frac{\mu_k}{\mu_0}(b+\sigma_k\mu_k \lambda_k+\bar{b}+\tilde{b}_k)\\
 &\ - \frac{\mu_k}{\mu_0} (\sigma_k \mu_k \lambda_k + b) - \mu_k \big(\sigma_k \frac{\bar{b}}{\mu_0} + (1-\sigma_k)(y_k-\lambda_k)\big)\\
 = &\ b + \frac{\mu_k}{\mu_0}(\bar{b}+\tilde{b}_k) - Ax_k - \mu_k (y_k-\lambda_k)\\
 = &\ 0, 
\end{split}
\end{equation*}
\noindent where the last equation follows from the neighbourhood conditions (i.e. $(x_k,y_k,z_k) \in \mathcal{N}_{\mu_k}(\zeta_k,\lambda_k)$). Similarly, we can show that:
$$-Q\bar{x} + A^T\bar{y} + \bar{z}-\mu_k \bar{x} = 0.$$
\par The previous two equalities imply that: 
\begin{equation} \label{Lemma boundedness Dx Dz, complementarity positivity}
\begin{split}
\bar{x}^T \bar{z} = 
\bar{x}^T(Q\bar{x} - A^T \bar{y}+\mu_k \bar{x}) =  \bar{x}^T (Q + \mu_k I)\bar{x} + \mu_k \bar{y}^T \bar{y} \geq 0.
\end{split}
\end{equation}
\noindent On the other hand, using the last block equation of the Newton system (\ref{Newton System}), we have:
\begin{equation*}
Z_k\bar{x} + X_k \bar{z} = - X_kZ_ke_n + \sigma_k \mu_k e_n + \frac{\mu_k}{\mu_0} Z_k(\tilde{x}-x_{r_k}^*)+\frac{\mu_k}{\mu_0}  X_k(\tilde{z}- z_{r_k}^*) + \mu_k Z_k \hat{x} + \mu_k X_k \hat{z}.
\end{equation*}
\noindent Let $W_k = (X_kZ_k)^{\frac{1}{2}}$. By multiplying both sides of the previous equation by $W_k^{-1}$, we get:
\begin{equation} \label{Lemma boundedness Dx Dz, relation 1}
D_k^{-1}\bar{x} + D_k\bar{z} = - W_k^{-1}(X_kZ_ke_n- \sigma_k \mu_ke_n) + \frac{\mu_k}{\mu_0} \big(D_k^{-1}(\tilde{x}-x_{r_k}^*) + D_k(\tilde{z}-z_{r_k}^*)\big)+ \mu_k \big(D_k^{-1} \hat{x} + D_k \hat{z}\big).
\end{equation}
\noindent But, from (\ref{Lemma boundedness Dx Dz, complementarity positivity}), we know that $\bar{x}^T \bar{z} \geq 0$ and hence:
\begin{equation*}
\|D_k^{-1}\bar{x} + D_k \bar{z}\|_2^2 \geq \|D_k^{-1} \bar{x}\|_2^2 + \|D_k \bar{z}\|_2^2.
\end{equation*}
\noindent Combining (\ref{Lemma boundedness Dx Dz, relation 1}) with the previous inequality, gives:
\begin{equation*}
\begin{split}
\|D_k^{-1}\bar{x}\|_2^2 + \|D_k\bar{z}\|_2^2  \leq &\ \bigg\{\|W_k^{-1}\|_2\|X_kZ_ke_n-\sigma_k \mu_k e_n\|_2 +\\ &\ \frac{\mu_k}{\mu_0} \big(\|D_k^{-1}(\tilde{x}-x_{r_k}^*)\|_2 +  \|D_k(\tilde{z}-z_{r_k}^*)\|_2\big) + \mu_k\big(\|D_k^{-1} \hat{x}\|_2+\| D_k \hat{z}\|_2 \big) \bigg\}^2.
\end{split}
\end{equation*}
\par We isolate one of the two terms of the left hand side of the previous inequality, take square roots on both sides, use \eqref{Lemma Dx Dz boundedness, auxiliary triple} and apply the triangular inequality to it, to obtain:
\begin{equation} \label{Lemma boundedness Dx Dz, relation 2}
\begin{split}
\|D_k^{-1} \Delta x_k \|_2 \leq &\ \|W_k^{-1}\|_2 \|X_kZ_ke_n -\sigma_k \mu_k e_n\|_2\\
&\ + \frac{\mu_k}{\mu_0}\big( 2\|D_k^{-1} (\tilde{x}-x_{r_k}^*)\|_2 +\|D_k(\tilde{z}-z_{r_k}^*)\|_2\big)+ \mu_k\big(2\|D_k^{-1} \hat{x}\|_2+\| D_k \hat{z}\|_2 \big).
\end{split}
\end{equation} 
\par We now proceed to bounding the terms in the right hand side of (\ref{Lemma boundedness Dx Dz, relation 2}). Firstly, notice from the neighbourhood conditions (\ref{Small neighbourhood}) that $\gamma_{\mu} \mu_k  \leq x^i_k z^i_k$. This in turn implies that:
\begin{equation*}
\|W_k^{-1}\|_2 = \max_i \frac{1}{(x^i_k z^i_k)^{\frac{1}{2}}} \leq \frac{1}{(\gamma_{\mu} \mu_k)^{\frac{1}{2}}}.
\end{equation*}
\noindent On the other hand, we have that:
\begin{equation*}
\begin{split}
\|X_kZ_ke_n-\sigma_k \mu_k e_n\|_2^2 = &\  \|X_kZ_ke_n\|^2 - 2\sigma_k \mu_k x_k^T z_k + \sigma_k^2\mu_k^2 n\\
\leq &\ \|X_kZ_ke_n\|_1^2 - 2\sigma_k \mu_k x_k^T z_k + \sigma_k^2\mu_k^2 n\\
= &\ (\mu_k n)^2 - 2\sigma_k \mu_k^2 n + \sigma_k \mu_k^2 n\\
\leq &\ \mu_k^2 n^2 .
\end{split}
\end{equation*}
\noindent Hence, combining the previous two relations yields:
\begin{equation*}
\|W_k^{-1}\|_2\|X_kZ_ke_n-\sigma_k \mu_k e_n\|_2 \leq \frac{n}{\gamma_{\mu}^{\frac{1}{2}}} \mu_k^{\frac{1}{2}} = O\big(n \mu^{\frac{1}{2}} \big).
\end{equation*}
\noindent We proceed by bounding $\|D_k^{-1}\|_2$. For that, using Lemma \ref{Lemma boundedness of x z}, we have:
\begin{equation*}
\|D_k^{-1}\|_2 = \max_i |(D_k^{ii})^{-1}| = \|D_k^{-1}e_n\|_{\infty} = \|W_k^{-1}Z_k e_n\|_{\infty} \leq \|W_k^{-1}\|_2\|z_k\|_1 = O\bigg(\frac{n}{\mu_k^{\frac{1}{2}}} \bigg).
\end{equation*}
\noindent Similarly, we have that:
\begin{equation*}
\|D_k\|_2 = O\bigg(\frac{n}{\mu_k^{\frac{1}{2}}} \bigg).
\end{equation*}
\noindent Hence, using the previous bounds, as well as Lemmas \ref{Lemma-boundedness of optimal solutions for sub-problems}, \ref{Lemma tilde point}, we obtain:
\begin{equation*}
\begin{split}
2\frac{\mu_k}{\mu_0}\|D_k^{-1}(\tilde{x}-x_{r_k}^*)\|_2 +\frac{\mu_k}{\mu_0}\|D_k(\tilde{z}-z_{r_k}^*)\|_2 \leq &\ 2\frac{\mu_k}{\mu_0}(\|D_k^{-1}\|_2 + \|D_k\|_2)\max\{\|\tilde{x}-x_{r_k}^*\|_2,\|\tilde{z}-z_{r_k}^*\|_2\}\\
= &\ O\big(n^{\frac{3}{2}}\mu_k^{\frac{1}{2}}\big),
\end{split}
\end{equation*}
\noindent and
\begin{equation*}
\begin{split}
\mu_k\big(2\|D_k^{-1} \hat{x}\|_2+\| D_k \hat{z}\|_2\big) \leq 2 \mu_k(\|D_k^{-1}\|_2 + \|D_k\|_2)\max\{\|\hat{x}\|_2,\|\hat{z}\|_2\} = O(n^2 \mu_k^{\frac{1}{2}}).
\end{split}
\end{equation*}
\noindent Combining all the previous bounds yields the claimed bound for $\|D_k^{-1}\Delta x_k\|_2$. One can bound $\|D_k \Delta z_k\|_2$ in the same way. The latter is omitted for ease of presentation. 
\par Finally, we have that:
$$\|\Delta x_k\|_2 = \|D_k D_k^{-1} \Delta x_k\|_2 \leq \|D_k\|_2\|D_k^{-1} \Delta x_k\|_2 = O(n^3).$$
\noindent Similarly, we can show that $\|\Delta z_k \|_2 = O(n^3)$. From the first block equation of the Newton system in \eqref{Newton System}, alongside Assumption \ref{Assumption 2}, we can show that $\|\Delta y_k\|_2 = O(n^3)$, which completes the proof. 
\end{proof}

\noindent We are now able to prove that at every iteration of Algorithm \ref{Algorithm PMM-IPM}, there exists a step-length $\alpha_k > 0$, using which, the new iterate satisfies the conditions required by the algorithm.
\begin{lemma} \label{Lemma step-length}
Given Assumptions \textnormal{\ref{Assumption 1}} and \textnormal{\ref{Assumption 2}}, there exists a step-length $\bar{\alpha} \in (0,1)$, such that for all $\alpha \in [0,\bar{\alpha}]$ and for all iterations $k \geq 0$ of Algorithm \textnormal{\ref{Algorithm PMM-IPM}}, the following relations hold:
\begin{equation} \label{Lemma step-length relation 1}
(x_k + \alpha \Delta x_k)^T(z_k + \alpha \Delta z_k) \geq (1-\alpha(1-\beta_1))x_k^T z_k,
\end{equation}
\begin{equation} \label{Lemma step-length relation 2}
(x^i_k + \alpha \Delta x^i_k)(z^i_k + \alpha \Delta z^i_k) \geq \frac{\gamma_{\mu}}{n}(x_k + \alpha \Delta x_k)^T(z_k + \alpha \Delta z_k),\ \text{for all } i \in \{1,\ldots,n\},
\end{equation}
\begin{equation} \label{Lemma step-length relation 3}
(x_k + \alpha \Delta x_k)^T(z_k + \alpha \Delta z_k) \leq (1-\alpha(1-\beta_2))x_k^T z_k,
\end{equation}
where, without loss of generality, $\beta_1 = \frac{\sigma_{\min}}{2}$ and $\beta_2 = 0.99$. Moreover, $\bar{\alpha} \geq \frac{\bar{\kappa}}{n^4}$ for all $k\geq 0$, where $\bar{\kappa} > 0$ is independent of $n$ and $m$, and if $(x_k,y_k,z_k) \in \mathcal{N}_{\mu_k}(\zeta_k,\lambda_k)$, then letting:
$$(x_{k+1},y_{k+1},z_{k+1}) = (x_k + \alpha\Delta x_k,y_k + \alpha\Delta y_k, z_k + \alpha\Delta z_k),\ \mu_{k+1} = \frac{x_{k+1}^Tz_{k+1}}{n},\ \ \forall\  \alpha \in (0,\bar{\alpha}]$$
\noindent gives: $(x_{k+1},y_{k+1},z_{k+1}) \in \mathcal{N}_{\mu_{k+1}}(\zeta_{k+1},\lambda_{k+1})$, where $\lambda_k,\ \zeta_k$ are updated as in Algorithm \textnormal{\ref{Algorithm PMM-IPM}}.
\end{lemma}
\begin{proof}
\par In order to prove the first three inequalities, we follow the developments in \cite[Chapter 6, Lemma 6.7]{book_4}. From Lemma \ref{Lemma boundedness Dx Dz}, we have that there exists a constant $C_{\Delta} >0$, such that:
$$(\Delta x_k)^T \Delta z_k = (D_k^{-1} \Delta x_k)^T (D_k \Delta z_k) \leq \|D_k^{-1} \Delta x_k\|_2 \|D_k \Delta z_k\|_2 \leq C_{\Delta}^2 n^4 \mu_k.$$
\noindent Similarly, it is easy to see that:
$$|\Delta x^i_k \Delta z^i_k| \leq C_{\Delta}^2 n^4 \mu_k.$$
\noindent On the other hand, by summing over all $n$ components of the last block equation of the Newton system (\ref{Newton System}), we have:
\begin{equation} \label{Lemma step-length equation 1}
z_k^T \Delta x_k + x_k^T \Delta z_k = e_n^T(Z_k \Delta x_k + X_k \Delta z_k) = e_n^T(-X_kZ_ke_n+ \sigma_k \mu_k e_n) = (\sigma_k - 1) x_k^T z_k,
\end{equation}
\noindent while the components of the last block equation of the Newton system (\ref{Newton System}) can be written as:
\begin{equation} \label{Lemma step-length equation 2}
z^i_k\Delta x^i_k + x^i_k \Delta z^i_k = -x^i_k z^i_k + \sigma_k \mu_k.
\end{equation}
\par We proceed by proving (\ref{Lemma step-length relation 1}). Using (\ref{Lemma step-length equation 1}), we have:
\begin{equation*}
\begin{split}
(x_k + \alpha \Delta x_k)^T (z_k + \alpha \Delta z_k) - (1-\alpha(1 -\beta_1))x_k^T z_k = \\
x_k^T z_k +\alpha (\sigma_k - 1)x_k^T z_k + \alpha^2 \Delta x_k^T \Delta z_k - (1-\alpha)x_k^T z_k -\alpha \beta_1 x_k^T z_k \geq \\
\alpha (\sigma_k - \beta_1) x_k^Tz_k  - \alpha^2 C_{\Delta}^2 n^4 \mu_k \geq \alpha (\frac{\sigma_{\min}}{2})n \mu_k - \alpha^2 C_{\Delta}^2 n^4 \mu_k,
\end{split}
\end{equation*}
\noindent where we set (without loss of generality) $\beta_1 = \frac{\sigma_{\min}}{2}$. The most-right hand side of the previous inequality will be non-negative for every $\alpha$ satisfying:
$$\alpha \leq \frac{\sigma_{\min}}{2 C_{\Delta}^2 n^3}.$$ 
\par In order to prove (\ref{Lemma step-length relation 2}), we will use (\ref{Lemma step-length equation 2}) and the fact that from the neighbourhood conditions we have that $x^i_k z^i_k \geq \gamma_{\mu} \mu_k$. In particular, we obtain:
\begin{equation*}
\begin{split}
(x^i_k + \alpha \Delta x^i_k)(z^i_k + \alpha \Delta z^i_k) \geq &\ (1-\alpha)x^i_k z^i_k + \alpha \sigma_k \mu_k - \alpha^2 C_{\Delta}^2 n^4 \mu_k \\
\geq &\ \gamma_{\mu}(1-\alpha)\mu_k + \alpha \sigma_k \mu_k - \alpha^2 C_{\Delta}^2 n^4 \mu_k.
\end{split}
\end{equation*}
\noindent By combining all the previous, we get:
\begin{equation*}
\begin{split}
(x^i_k + \alpha \Delta x^i_k)(z^i_k + \alpha \Delta z^i_k) - \frac{\gamma_{\mu}}{n}(x_k + \alpha \Delta x_k)^T(z_k + \alpha \Delta z_k) \geq \\
\alpha \sigma_k (1-\gamma_{\mu})\mu_k - (1 + \frac{\gamma_{\mu}}{n})\alpha^2 C_{\Delta}^2 n^4\mu_k \geq \\ \alpha\sigma_{\min}(1-\gamma_{\mu})\mu_k - 2\alpha^2 C_{\Delta}^2 n^4\mu_k.
\end{split}
\end{equation*}
\noindent In turn, the most-right hand side of the previous relation will be non-negative for every $\alpha$ satisfying:
$$\alpha \leq \frac{\sigma_{\min}(1-\gamma_{\mu})}{2C_{\Delta}^2 n^4}.$$
\par Finally, to prove (\ref{Lemma step-length relation 3}), we set (without loss of generality) $\beta_2 = 0.99$. We know, from Algorithm \ref{Algorithm PMM-IPM}, that $\sigma_{\max} \leq 0.5$. With the previous two remarks in mind, we have:
\begin{equation*}
\begin{split}
\frac{1}{n}(x_k + \alpha \Delta x_k)^T (z_k + \alpha \Delta z_k) - (1-0.01\alpha)\mu_k \leq \\
(1-\alpha)\mu_k + \alpha \sigma_k \mu_k + \alpha^2 \frac{C_{\Delta}^2 n^4}{n}\mu_k - (1-0.01 \alpha)\mu_k \leq \\
-0.99\alpha \mu_k + 0.5\alpha \mu_k + \alpha^2 \frac{C_{\Delta}^2 n^4}{n} \mu_k =\\
-0.49\alpha \mu_k +\alpha^2\frac{C_{\Delta}^2 n^4}{n}\mu_k.
\end{split}
\end{equation*}
\noindent The last term will be non-positive for every $\alpha$ satisfying:
$$\alpha \leq \frac{0.49 }{C_{\Delta}^2 n^3}.$$
\par By combining all the previous bounds on the step-length, we have that \eqref{Lemma step-length relation 1}-\eqref{Lemma step-length relation 3} will hold for every $\alpha \in (0,\alpha^*)$, where:
\begin{equation} \label{Lemma step-length bound on step-length}
\alpha^* = \min\big\{ \frac{\sigma_{\min}}{2 C_{\Delta}^2 n^3},\ \frac{\sigma_{\min}(1-\gamma_{\mu})}{2C_{\Delta}^2 n^4},\ \frac{0.49 }{C_{\Delta}^2 n^3},\ 1\big\}.
\end{equation}
\par Next, we would like to find the maximum $\bar{\alpha} \in (0,\alpha^*]$, such that: 
$$(x_k(\alpha),y_k(\alpha),z_k(\alpha)) \in \mathcal{N}_{\mu_k(\alpha)}(\zeta_k,\lambda_k),\ \forall\ \alpha \in (0,\bar{\alpha}],$$
\noindent where $\mu_k(\alpha) = \frac{x_k(\alpha)^T z_k(\alpha)}{n}$ and:
$$(x_k(\alpha),y_k(\alpha),z_k(\alpha)) = (x_k + \alpha \Delta x_k,y_k + \alpha \Delta y_k,z_k + \alpha \Delta z_k).$$
\noindent Let:
$$\tilde{r}_p(\alpha) = Ax_k(\alpha)+ \mu_k(\alpha)(y_k(\alpha)-\lambda_k) - \big(b + \frac{\mu_k(\alpha)}{\mu_0}\bar{b}\big),$$
\noindent and
$$\tilde{r}_d(\alpha) = -Qx_k(\alpha) + A^Ty_k(\alpha) + z_k(\alpha) - \mu_k(\alpha)(x_k(\alpha)- \zeta_k) - \big(c + \frac{\mu_k(\alpha)}{\mu_0}\bar{c}\big).$$
\noindent In other words, we need to find the maximum $\bar{\alpha} \in (0,\alpha^*]$, such that:
\begin{equation} \label{Step-length neighbourhood conditions}
\|\tilde{r}_p(\alpha),\tilde{r}_d(\alpha)\|_2 \leq C_N \frac{\mu_k(\alpha)}{\mu_0},\ \ \|\tilde{r}_p(\alpha),\tilde{r}_d(\alpha)\|_{\mathcal{A}} \leq \gamma_{\mathcal{A}} \rho\frac{\mu_k(\alpha)}{\mu_0} ,\ \forall\ \alpha \in (0,\bar{\alpha}].
\end{equation}
\noindent If the latter two conditions hold, then $(x_k(\alpha),y_k(\alpha),z_k(\alpha)) \in \mathcal{N}_{\mu_k(\alpha)}(\zeta_k,\lambda_k),\ \forall\ \alpha \in (0,\bar{\alpha}]$. Then, if Algorithm \ref{Algorithm PMM-IPM} updates $\zeta_k,\ \lambda_k$, it does so only when similar conditions (as in \eqref{Step-length neighbourhood conditions}) hold for the new parameters. If the parameters are not updated, then the new iterate lies in the desired neighbourhood because of \eqref{Step-length neighbourhood conditions}, alongside \eqref{Lemma step-length relation 1}-\eqref{Lemma step-length relation 3}.
\par We start by rearranging $\tilde{r}_p(\alpha)$. Specifically, we have that:\begin{equation*} 
\begin{split}
\tilde{r}_p(\alpha) = &\ A(x_k + \alpha \Delta x_k) +\big(\mu_k + \alpha(\sigma_k-1)\mu_k + \alpha^2\frac{\Delta x_k^T \Delta z_k}{n}\big)\big((y_k + \alpha \Delta y_k -\lambda_k)-\frac{\bar{b}}{\mu_0} \big) -b \\
= &\ \big(Ax_k +\mu_k (y_k -\lambda_k)-b -\frac{\mu_k}{\mu_0}\bar{b}\big) + \alpha(A \Delta x_k + \mu_k \Delta y_k)\ + \\
&\ +\ \big(\alpha(\sigma_k-1)\mu_k + \alpha^2\frac{\Delta x_k^T \Delta z_k}{n}\big)\big((y_k - \lambda_k + \alpha \Delta y_k) - \frac{\bar{b}}{\mu_0}\big)\\
= &\ \frac{\mu_k}{\mu_0}\tilde{b}_k + \alpha\bigg(b- Ax_k - \sigma_k\mu_k\big((y_k-\lambda_k)-\frac{\bar{b}}{\mu_0} \big) + \mu_k \big((y_k-\lambda_k)-\frac{\bar{b}}{\mu_0} \big)\ - \\
 &\ -\ \mu_k \big((y_k-\lambda_k)-\frac{\bar{b}}{\mu_0} \big) \bigg) + \big(\alpha(\sigma_k-1)\mu_k + \alpha^2\frac{\Delta x_k^T \Delta z_k}{n}\big)\big((y_k - \lambda_k + \alpha \Delta y_k) - \frac{\bar{b}}{\mu_0}\big),
\end{split}
\end{equation*}
\noindent where we used that $\mu_k(\alpha) = \big(\mu_k + \alpha(\sigma_k-1)\mu_k + \alpha^2 \frac{\Delta x_k^T \Delta z_k}{n}\big)$, which can be derived from \eqref{Lemma step-length equation 1}, as well as the neighbourhood conditions \eqref{Small neighbourhood}, and the second block equation of the Newton system \eqref{Newton System}. By using again the neighbourhood conditions, and then by deleting the opposite terms in the previous equation, we obtain:\begin{equation}\label{primal infeasibility formula}
\begin{split}
\tilde{r}_p(\alpha) = &\ (1-\alpha)\frac{\mu_k}{\mu_0}\tilde{b}_k + \alpha^2(\sigma_k - 1)\mu_k \Delta y_k + \alpha^2\frac{\Delta x_k^T \Delta z_k}{n}\big(y_k - \lambda_k + \alpha \Delta y_k - \frac{\bar{b}}{\mu_0} \big).
\end{split}
\end{equation}
\noindent Similarly, we can show that:
\begin{equation}\label{dual infeasibility formula}
\tilde{r}_d(\alpha) = (1-\alpha)\frac{\mu_k}{\mu_0}\tilde{c}_k - \alpha^2(\sigma_k-1)\mu_k \Delta x_k - \alpha^2\frac{\Delta x_k^T \Delta z_k}{n}\big(x_k - \zeta_k + \alpha \Delta x_k + \frac{\bar{c}}{\mu_0} \big).
\end{equation}
\par Define the following two quantities:
\begin{equation} \label{step-length, auxiliary constants}
\begin{split}
\xi_2 = &\ \mu_k \|(\Delta y_k,\Delta x_k)\|_2 + C_{\Delta}^2n^3 \mu_{k}\bigg(\|(y_k - \lambda_k,x_k-\zeta_k)\|_2 + \alpha^* \|(\Delta y_k,\Delta x_k)\|_2 + \bigg\|\bigg(\frac{\bar{b}}{\mu_0},\frac{\bar{c}}{\mu_0}\bigg)\bigg\|_2\bigg),\\
\xi_{\mathcal{A}} = &\ \mu_k \|(\Delta y_k,\Delta x_k)\|_{\mathcal{A}} + C_{\Delta}^2n^3 \mu_{k}\bigg(\|(y_k - \lambda_k,x_k-\zeta_k)\|_{\mathcal{A}} + \alpha^* \|(\Delta y_k,\Delta x_k)\|_{\mathcal{A}} + \bigg\|\bigg(\frac{\bar{b}}{\mu_0},\frac{\bar{c}}{\mu_0}\bigg)\bigg\|_{\mathcal{A}}\bigg),
\end{split}
\end{equation}
\noindent where $\alpha^*$ is given by \eqref{Lemma step-length bound on step-length}. Using the definition of the starting point in \eqref{starting point}, as well as results in Lemmas \ref{Lemma boundedness of x z}, \ref{Lemma boundedness Dx Dz}, we can observe that $\xi_2 = O(n^4 \mu_k)$. On the other hand, using Assumption \ref{Assumption 2}, we know that for every pair $(r_1,r_2) \in \mathbb{R}^{m+n}$, if $\|(r_1,r_2)\|_2 = \Theta(f(n))$, where $f(\cdot)$ is a positive polynomial function of $n$, then $\|(r_1,r_2)\|_{\mathcal{A}} = \Theta(f(n))$. In other words, we have that $\xi_{\mathcal{A}} = O(n^4 \mu_k)$. Using the quantities in \eqref{step-length, auxiliary constants}, equations \eqref{primal infeasibility formula}, \eqref{dual infeasibility formula}, as well as the neighbourhood conditions, we have that:
\begin{equation*}
\begin{split}
\|\tilde{r}_p(\alpha),\tilde{r}_d(\alpha)\|_2 \leq &\ (1-\alpha)C_N \frac{\mu_k}{\mu_0} + \alpha^2 \mu_k \xi_2,\\
\|\tilde{r}_p(\alpha),\tilde{r}_d(\alpha)\|_A \leq &\  (1-\alpha)\gamma_{\mathcal{A}}\rho \frac{\mu_k}{\mu_0} + \alpha^2 \mu_k \xi_{\mathcal{A}},
\end{split}
\end{equation*}
\noindent for all $\alpha \in (0,\alpha^*]$, where $\alpha^*$ is given by \eqref{Lemma step-length bound on step-length}. On the other hand, we know from \eqref{Lemma step-length relation 1}, that: 
$$\mu_k(\alpha) \geq (1-\alpha(1-\beta_1))\mu_k,\ \forall\ \alpha \in (0,\alpha^*].$$
\noindent By combining the last two inequalities, we get that:
\begin{equation*}
\begin{split}
\|\tilde{r}_p(\alpha),\tilde{r}_d(\alpha)\|_2 \leq \frac{\mu_k(\alpha)}{\mu_0} C_N,\ \forall\ \alpha \in \bigg(0, \min\big\{\alpha^*,\frac{\beta_1 C_N}{\xi_2 \mu_0}\big\}\bigg].
\end{split}
\end{equation*}
\noindent Similarly, 
\begin{equation*}
\begin{split}
\|\tilde{r}_p(\alpha),\tilde{r}_d(\alpha)\|_{\mathcal{A}} \leq \frac{\mu_k(\alpha)}{\mu_0} \gamma_{\mathcal{A}} \rho,\ \forall\ \alpha \in \bigg(0, \min \big\{\alpha^*,\frac{\beta_1 \gamma_{\mathcal{A}} \rho}{\xi_{\mathcal{A}} \mu_0}\big\}\bigg].
\end{split}
\end{equation*}
\par Hence, we have that:
\begin{equation} \label{Lemma step-length, STEPLENGTH BOUND}
\bar{\alpha} = \min \bigg\{\alpha^*,\frac{\beta_1 C_N}{\xi_2 \mu_0}, \frac{\beta_1 \gamma_{\mathcal{A}} \rho}{\xi_{\mathcal{A}} \mu_0} \bigg\},
\end{equation}
\noindent where $\beta_1 = \frac{\sigma_{\min}}{2}$. Since $\bar{\alpha} = \Omega\big(\frac{1}{n^4}\big)$, we know that there must exist a constant $\bar{\kappa} > 0$, independent of $n$, $m$ and of the iteration $k$, such that: $\bar{\alpha} \geq \frac{\kappa}{n^4}$, for all $k \geq 0$, and this completes the proof.
\end{proof}
\noindent The following Theorem summarizes our results.

\begin{theorem} \label{Theorem mu convergence}
Given Assumptions \textnormal{\ref{Assumption 1}, \ref{Assumption 2}}, the sequence $\{\mu_k\}$ generated by Algorithm \textnormal{\ref{Algorithm PMM-IPM}} converges Q-linearly to zero, and the sequences of regularized residual norms 
$$\{\|Ax_k + \mu_k (y_k-\lambda_k) - b-\frac{\mu_k}{\mu_0}\bar{b}\|_2\}\ \text{and}\  \{\|-Qx_k + A^T y_k + z_k - \mu_k(x_k - \zeta_k) - c - \frac{\mu_k}{\mu_0}\bar{c}\|_2\}$$
converge R-linearly to zero.
\end{theorem}
\begin{proof}
\noindent From (\ref{Lemma step-length relation 3}) we have that:
$$ \mu_{k+1} \leq (1-0.01\alpha_k)\mu_k,$$
\noindent while, from (\ref{Lemma step-length, STEPLENGTH BOUND}), we know that $\forall\ k \geq 0$, $\exists\ \bar{\alpha} \geq \frac{\bar{\kappa}}{n^4}$ such that: $\alpha_k \geq \bar{\alpha}$. Hence, we can easily see that $\mu_k \rightarrow 0$. On the other hand, from the neighbourhood conditions, we know that for all $k \geq 0$:
$$ \|Ax_k + \mu_k (y_k-\lambda_k) - b - \frac{\mu_k}{\mu_0}\bar{b}\|_2 \leq C_N \frac{\mu_k}{\mu_0}$$
\noindent and
$$\|-Qx_k + A^T y_k + z_k - \mu_k(x_k - \zeta_k) - c- \frac{\mu_k}{\mu_0}\bar{c}\|_2 \leq C_N \frac{\mu_k}{\mu_0}.$$
\noindent This completes the proof. 
\end{proof}

\begin{theorem} \label{Theorem complexity}
\noindent Let $\epsilon \in (0,1)$ be a given error tolerance. Choose a starting point for Algorithm \textnormal{\ref{Algorithm PMM-IPM}} as in \eqref{starting point}, such that $\mu_0 \leq \frac{C}{\epsilon^{\omega}}$ for some positive constants $C,\ \omega$. Given Assumptions \textnormal{\ref{Assumption 1}} and \textnormal{\ref{Assumption 2}}, there exists an index $K$ with:
$$K = O(n^4|\log\big(\frac{1}{\epsilon}\big)|\big)$$
\noindent such that the iterates $\{w_k\} = \{(x_k^T,y_k^T,z_k^T)^T\}$ generated from Algorithm \textnormal{\ref{Algorithm PMM-IPM}} satisfy:
$$\mu_k \leq \epsilon,\ \ \ \ \forall\ k\geq K.$$  
\end{theorem}
\begin{proof}
%\iffalse 
\noindent The proof follows the developments in \cite{book_4,paper_30} and is only provided here for completeness. Without loss of generality, we can chose $\sigma_{\max} \leq 0.5$ and then from Lemma \ref{Lemma step-length}, we know that there is a constant $\bar{\kappa}$ independent of $n$ such that: $\bar{a} \geq \frac{\bar{\kappa}}{n^4}$, where $\bar{a}$ is the worst-case step-length. Given the latter, we know that the new iterate lies in the neighbourhood $\mathcal{N}_{\mu_{k+1}}(\zeta_{k+1},\lambda_{k+1})$ defined in (\ref{Small neighbourhood}). We also know, from (\ref{Lemma step-length relation 3}), that:
$$\mu_{k+1} \leq (1 - 0.01\bar{a})\mu_k \leq (1-0.01\frac{\bar{\kappa}}{n^4})\mu_k,\ \ \ \  k = 0,1,2,\ldots$$
\noindent By taking logarithms on both sides in the previous inequality, we get:
$$\log (\mu_{k+1}) \leq \log (1 - \frac{\tilde{\kappa}}{n^4}) + \log (\mu_k),$$
\noindent where $\tilde{\kappa} = 0.01\bar{\kappa}$. By applying repeatedly the previous formula, and using the fact that $\mu_0 \leq \frac{C}{\epsilon^{\omega}}$, we have:
$$\log(\mu_k) \leq k \log(1 - \frac{\tilde{\kappa}}{n^4}) + \log(\mu_0) \leq k \log(1 - \frac{\tilde{\kappa}}{n^4}) + \omega \log(\frac{1}{\epsilon}) + \log(C).$$
\noindent We use the fact that: $\log(1+\beta) \leq \beta,\ \ \forall\ \beta > -1$ to get:
$$\log(\mu_k) \leq k(-\frac{\tilde{\kappa}}{n^4}) + \omega \log(\frac{1}{\epsilon}) +\log(C).$$
\noindent Hence, convergence is attained if:
$$k(-\frac{\tilde{\kappa}}{n^4}) + \omega \log(\frac{1}{\epsilon}) +\log(C) \leq \log(\epsilon).$$
\noindent The latter holds for all $k$ satisfying:
$$k \geq K = \frac{n^4}{\tilde{\kappa}}((1+\omega)\log(\frac{1}{\epsilon})+\log(C)),$$
\noindent which completes the proof. 
%\fi
\end{proof}
\par Finally, we present the global convergence guarantee of Algorithm \ref{Algorithm PMM-IPM}.
\begin{theorem} \label{Theorem convergence for the feasible case}
Suppose that Algorithm \textnormal{\ref{Algorithm PMM-IPM}} terminates when a limit point is reached. Then, if Assumptions \textnormal{\ref{Assumption 1}} and \textnormal{\ref{Assumption 2}} hold, every limit point of $\{(x_k,y_k,z_k)\}$ determines a primal-dual solution of the non-regularized pair \textnormal{(\ref{non-regularized primal})-(\ref{non-regularized dual})}.
\end{theorem}
\begin{proof}
\noindent From Theorem \ref{Theorem mu convergence}, we know that $\{\mu_k\} \rightarrow 0$, and hence, there exists a sub-sequence $\mathcal{K} \subseteq \mathbb{N}$, such that:
$$\{ Ax_k + \mu_k (y_k - \lambda_k) -b -\frac{\mu_k}{\mu_0}\bar{b}\}_{\mathcal{K}} \rightarrow 0,\ \{-Qx_k + A^T y_k + z_k - \mu_k (x_k - \zeta_k)-c-\frac{\mu_k}{\mu_0}\bar{c}\}_{\mathcal{K}} \rightarrow 0.$$
\noindent However, since Assumptions \ref{Assumption 1} and \ref{Assumption 2} hold, we know from Lemma \ref{Lemma boundedness of x z} that $\{(x_k,y_k,z_k)\}$ is a bounded sequence. Hence, we obtain that:
$$\{ Ax_k - b\}_{\mathcal{K}} \rightarrow 0,\ \{-Qx_k + A^Ty_k +z_k -c\}_{\mathcal{K}} \rightarrow 0.$$
\noindent One can readily observe that the limit point of the algorithm satisfies the conditions given in (\ref{non-regularized FOC}), since $\mu_k = \frac{x_k^Tz_k}{n}$. 
\end{proof}

\section{Infeasible Problems} \label{section Infeasible problems}
\par Let us now drop Assumptions \ref{Assumption 1}, \ref{Assumption 2}, in order to analyze the behaviour of the algorithm in the case where an infeasible problem is tackled. Let us employ the following two premises:
\begin{premise} \label{Premise 1}
During the iterations of Algorithm \textnormal{\ref{Algorithm PMM-IPM}}, the sequences $\{\|y_k - \lambda_k\|_2\}$ and $\{\|x_k - \zeta_k\|_2\}$, remain bounded.
\end{premise}

\begin{premise} \label{Premise 2}
There does not exist a primal-dual triple, satisfying the KKT conditions for the primal-dual pair \textnormal{(\ref{non-regularized primal})-(\ref{non-regularized dual})}.
\end{premise}
\par The following analysis is based on the developments in \cite{paper_2} and \cite{paper_15}. However, in these papers such an analysis is proposed in order to derive convergence of an IPM, while here, we use it as a tool in order to construct a reliable and implementable infeasibility detection mechanism. In what follows, we show that Premises \ref{Premise 1} and \ref{Premise 2} are contradictory. In other words, if Premise \ref{Premise 2} holds (which means that our problem is infeasible), then we will show that Premise \ref{Premise 1} cannot hold, and hence the negation of Premise \ref{Premise 1} is a necessary condition for infeasibility.

\begin{lemma} \label{Lemma infeasibility bounded Newton}
Given Premise \textnormal{\ref{Premise 1}}, and by assuming that $x_k^Tz_k > \epsilon$, for some $\epsilon >0$, for all iterations $k$ of Algorithm \textnormal{\ref{Algorithm PMM-IPM}}, the Newton direction produced by \textnormal{(\ref{Newton System})} is uniformly bounded by a constant dependent only on $n$. 
\end{lemma}
\begin{proof}
\par Let us use a variation of Theorem 1 given in \cite{paper_15}. This theorem states that if the following conditions are satisfied, 
\begin{enumerate}
\item $\mu_k > 0$, 
\item $\exists$ $\ $ $\bar{\epsilon} > 0 \ :\ x_k^i  z_k^i \geq \bar{\epsilon},\ \forall\ i=\{1,2,...,n\}, \forall\ k\geq 0$,
\item and the matrix $H_k = \mu_k I + Q + X_k^{-1}Z_k +\frac{1}{\mu_k}A^T A$ is positive definite,
\end{enumerate}
\noindent then the Jacobian matrix in (\ref{Newton System}) is non-singular and has a uniformly bounded inverse. Note that (1.), (3.) are trivial to verify, based on the our assumption that $x_k^T z_k = n \mu_k > \epsilon$. Condition (2.) follows since we know that our iterates lie in $\mathcal{N}_{\mu_k}(\zeta_k,\lambda_k)$, while we have $x_k^T z_k > \epsilon$, by assumption. Indeed, from the neighbourhood conditions (\ref{Small neighbourhood}), we have that $x_k^i z_k^i \geq \frac{\gamma_{\mu}}{n} x_k^T z_k$. Hence, there exists $\bar{\epsilon} = \frac{\gamma_{\mu} \epsilon}{n} > 0$ such that $x_k^i z_k^i > \bar{\epsilon},\ \forall \ k \geq 0,\ \forall \  i = \{1,\cdots,n\}$. 
\par Finally, we have to show that the right hand side of \eqref{Newton System} is uniformly bounded. To that end, we bound the right-hand side of the second block equation of \eqref{Newton System} as follows:
\begin{equation*}
\begin{split}
 \bigg\|Ax_k + \sigma_k \mu_k (y_k - \lambda_k) - b - \frac{\sigma_k \mu_k}{\mu_0}\bar{b} + \mu_k (y_k - \lambda_k -\frac{\bar{b}}{\mu_0}) -\mu_k (y_k - \lambda_k -\frac{\bar{b}}{\mu_0})\bigg\|_2 \leq \\
 \frac{\mu_k}{\mu_0}\|\tilde{b}_k\|_2 + \mu_k\bigg\|y_k - \lambda_k - \frac{\bar{b}}{\mu_0}\bigg\|_2,
\end{split}
\end{equation*}
\noindent where we used the neighbourhood conditions (\ref{Small neighbourhood}). Boundedness follows from Premise \ref{Premise 1}. A similar reasoning applies for bounding the right-hand side of the first block equation, while the right-hand side of the third block equation is bounded directly from the neighbourhood conditions. Combining the previous completes the proof. 
 \end{proof}
 
\par In the following Lemma, we prove by contradiction that the parameter $\mu_k$ of Algorithm \ref{Algorithm PMM-IPM} converges to zero, given that Premise \ref{Premise 1} holds. The proof is based on the developments in \cite{paper_41,paper_2} and is only partially given here, for ease of presentation.

\begin{lemma} \label{Lemma infeasibility mu to zero}
Given Premise \textnormal{\ref{Premise 1}}, and a sequence $(x_k,y_k,z_k) \in \mathcal{N}_{\mu_k}(\zeta_k,\lambda_k)$ produced by Algorithm \textnormal{\ref{Algorithm PMM-IPM}}, the sequence $\{\mu_k\}$ converges to zero.
\end{lemma}
\begin{proof}
\noindent Assume, by virtue of contradiction, that $\mu_k > \epsilon$, $\forall\ k \geq 0$. Then, we know that the Newton direction obtained by the algorithm at every iteration, after solving (\ref{Newton System}), will be uniformly bounded by a constant dependent only on $n$, that is, there exist positive constants $C^{\dagger}_1, C_2^{\dagger}$, such that $\|(\Delta x_k,\Delta y_k,\Delta z_k)\|_2 \leq C^{\dagger}_1$ and $\|(\Delta x_k,\Delta y_k,\Delta z_k)\|_{\mathcal{A}} \leq C_2^{\dagger}$. As in Lemma \ref{Lemma step-length}, we define:
$$\tilde{r}_p(\alpha) = Ax_k(\alpha) + \mu_k(\alpha) (y_k(\alpha) - \lambda_{k}) - b - \frac{\mu_k(\alpha)}{\mu_0}\bar{b},$$
\noindent and 
$$\tilde{r}_d(\alpha) = - Qx_k(\alpha) + A^T y_k(\alpha) + z_k(\alpha) - \mu_k(\alpha) (x_k(\alpha) - \zeta_{k}) - c - \frac{\mu_k(\alpha)}{\mu_0}\bar{c},$$
\noindent for which we know that equalities \eqref{primal infeasibility formula} and \eqref{dual infeasibility formula} hold, respectively.
 Take any $k \geq 0$ and define the following functions:
\begin{equation*}
\begin{split}
f_1(\alpha) = \ &(x_k+\alpha\Delta x_k)^T(z_k + \alpha \Delta z_k) - (1 -\alpha(1-\frac{\sigma_{\min}}{2}))x_k^T z_k,\\
f_2^i(\alpha) =\ & (x_k^i+ \alpha \Delta x_k^i)(z_k^i + \alpha \Delta z_k^i) - \gamma_{\mu} \mu_k(\alpha),\ i=1,\cdots,n,\\
f_3(\alpha) = \ & (1-0.01\alpha)x_k^Tz_k - (x_k(\alpha))^T (z_k(\alpha)),\\
g_{2}(\alpha) = \ & \frac{\mu_k(\alpha)}{\mu_0}C_N - \|(\tilde{r}_p(\alpha),\tilde{r}_d(\alpha))\|_2, \\
g_{\mathcal{A}}(\alpha) = \ & \frac{\mu_k(\alpha)}{\mu_0}\gamma_{\alpha} \rho - \|(\tilde{r}_p(\alpha),\tilde{r}_d(\alpha))\|_{\mathcal{A}},
\end{split}
\end{equation*}
\noindent where $\mu_k(\alpha) = \frac{(x_k + \alpha \Delta x_k)^T (z_k + \alpha \Delta z_k)}{n}$, $(x_k(\alpha),y_k(\alpha),z_k(\alpha)) = (x_k + \alpha \Delta x_k,y_k + \alpha \Delta y_k,z_k + \alpha \Delta z_k)$. We would like to show that there exists $\alpha^* > 0$, such that:
$$f_1(\alpha) \geq 0,\ f_2^i(\alpha) \geq 0,\ \forall\ i = 1,\ldots,n,\ f_3(\alpha) \geq 0,\ g_2(\alpha) \geq 0,\ g_{\mathcal{A}}(\alpha) \geq 0,\ \forall\ \alpha \in (0,\alpha^*].$$
\noindent These conditions model the requirement that the next iteration of Algorithm \ref{Algorithm PMM-IPM} must lie in the updated neighbourhood: $\mathcal{N}_{\mu_{k+1}}(\zeta_k,\lambda_{k})$. Note that Algorithm \ref{Algorithm PMM-IPM} updates the parameters $\lambda_k,\ \zeta_k$ only if the selected new iterate belongs to the new neighbourhood, defined using the updated parameters. Hence, it suffices to show that $(x_{k+1},y_{k+1},z_{k+1}) \in \mathcal{N}_{\mu_{k+1}}(\zeta_k,\lambda_{k})$. 
\par Proving the existence of $\alpha^* > 0$, such that each of the aforementioned functions is positive, follows exactly the developments in Lemma \ref{Lemma step-length}, with the only difference being that the bounds on the directions are not explicitly specified in this case. Using the same methodology as in Lemma \ref{Lemma step-length}, while keeping in mind our assumption, namely $x_k^T z_k > \epsilon$, and hence $ x_k^i z_k^i > \bar{\epsilon}$, we can show that:

\begin{equation} \label{infeasibility minimum step-length}
\alpha^* = \min \bigg\{1,\frac{\sigma_{\min}\epsilon}{2(C_1^{\dagger})^2}, \frac{(1-\gamma_{\mu})\sigma_{\min}\bar{\epsilon}}{2(C_1^{\dagger})^2},\frac{0.49\epsilon}{2(C_1^{\dagger})^2}, \frac{\sigma_{\min}C_N \epsilon}{2\mu_0(\xi_2)},\frac{\sigma_{\min}\gamma_{\mathcal{A}}\rho \epsilon}{2\xi_{\mathcal{A}}\mu_0} \bigg\},
\end{equation}
\noindent where $\xi_2,\ \xi_{\mathcal{A}}$ are bounded constants, defined as in \eqref{step-length, auxiliary constants}, and dependent on $C_1^{\dagger},\ C_2^{\dagger}$. However, using the inequality:
$$\mu_{k+1} \leq (1-0.01 \alpha)\mu_k,\ \forall\ \alpha \in [0,\alpha^*]$$
\noindent we get that $\mu_k \rightarrow 0$, which contradicts our assumption that $\mu_k > \epsilon,\ \forall\ k\geq 0$, and completes the proof. 
\end{proof}
\noindent Finally, using the following Theorem, we derive a necessary condition for infeasibility.
\begin{theorem} \label{Theorem Infeasibility condition}
Given Premise \textnormal{\ref{Premise 2}}, i.e. there does not exist a KKT triple for the pair \textnormal{(\ref{non-regularized primal})-(\ref{non-regularized dual})}, then Premise \textnormal{\ref{Premise 1}} fails to hold.
\end{theorem}
\begin{proof}
\noindent By virtue of contradiction, let Premise \ref{Premise 1} hold. In Lemma \ref{Lemma infeasibility mu to zero}, we proved that given Premise \ref{Premise 1}, Algorithm \ref{Algorithm PMM-IPM} produces iterates that belong to the neighbourhood (\ref{Small neighbourhood}) and $\mu_k \rightarrow 0$. But from the neighbourhood conditions we can observe that:
$$\|Ax_k + \mu_k(y_k - \lambda_k) - b - \frac{\mu_k}{\mu_0}\bar{b} \|_2 \leq C_N\frac{\mu_k}{\mu_0},$$
\noindent and
$$\|-Qx_k + A^Ty_k + z_k - \mu_k(x_k - \zeta_k)-c-\frac{\mu_k}{\mu_0}\bar{c}\|_2 \leq C_N \frac{\mu_k}{\mu_0}.$$
\noindent Hence, we can choose a sub-sequence $\mathcal{K} \subseteq \mathbb{N}$, for which:
$$\{Ax_k + \mu_k(y_k - \lambda_k) - b - \frac{\mu_k}{\mu_0}\bar{b} \}_{\mathcal{K}} \rightarrow 0,\ \text{and} \ \{-Qx_k + A^Ty_k + z_k - \mu_k(x_k - \zeta_k)-c-\frac{\mu_k}{\mu_0}\bar{c}\}_{\mathcal{K}} \rightarrow 0.$$
\noindent But since $\|y_k-\lambda_k\|_2$ and $\|x_k - \zeta_k\|_2$ are bounded, while $\mu_k \rightarrow 0$, we have that:
$$\{Ax_k - b\}_{\mathcal{K}} \rightarrow 0,\ \{c + Qx_k - A^T y_k - z_k\}_{\mathcal{K}} \rightarrow 0,\ \text{and}\ \{x_k^T z_k\}_{\mathcal{K}} \rightarrow 0.$$
\noindent This contradicts Premise \ref{Premise 2}, i.e. that the pair (\ref{non-regularized primal})-(\ref{non-regularized dual}) does not have a KKT triple, and completes the proof. 
\end{proof}
\par In the previous Theorem, we proved that Premise \ref{Premise 1} is a necessary condition for infeasibility, since otherwise, we arrive at a contradiction. Nevertheless, this does not mean that the condition is also sufficient. In order to obtain a more reliable test for infeasibility, that uses the previous result, we will have to use the properties of Algorithm \ref{Algorithm PMM-IPM}. In particular, we can notice that if the primal-dual problem is infeasible, then the PMM sub-problem will stop being updated after a finite number of iterations. In that case, we know from Theorem \ref{Theorem Infeasibility condition} that the sequence $\|(x_k-\zeta_k,y_k - \lambda_k)\|_2$ will grow unbounded. Hence, we can define a maximum number of iterations per PMM sub-problem, say $K^{\dagger} \geq 0$, as well as a very large constant $C^{\dagger} \geq 0$. Then, if $\|(x_k-\zeta_k,y_k - \lambda_k)\|_2 > C^{\dagger}$ and $k \geq K^{\dagger}$, the algorithm is terminated. The specific choices for these constants will be given in the next section.

\section{Computational Experience} \label{section numerical results}

\par In this section, we provide some implementation details and present computational results of the method, over a set of small to large scale linear and convex quadratic programming problems. The code was written in MATLAB and can be found in the following link:\\ \centerline{\textbf{https://www.maths.ed.ac.uk/ERGO/software.html} (\href{https://github.com/spougkakiotis/IP_PMM}{source link}).}
\subsection{Implementation Details}

\par Our implementation deviates from the theory, in order to gain some additional control, as well as computational efficiency. Nevertheless, the theory has served as a guideline to tune the code reliably. There are two major differences between the practical implementation of IP-PMM and its theoretical algorithmic counterpart. Firstly, our implementation uses different penalty parameters for the proximal terms and the logarithmic barrier term. In particular, we define a primal proximal penalty $\rho$, a dual proximal penalty $\delta$ and the barrier parameter $\mu$. Using the previous, the PMM Lagrangian function in \eqref{PMM lagrangian}, at an arbitrary iteration $k$ of the algorithm, becomes:
\begin{equation*} 
\begin{split}
\mathcal{L}^{PMM}_{\mu_k,\delta_k,\rho_k} (x;\zeta_k, \lambda_k) = c^Tx + \frac{1}{2}x^T Q x -\lambda_k^T (Ax - b) + \frac{1}{2\delta_k}\|Ax-b\|_2^2 + \frac{\rho_k}{2}\|x-\zeta_k\|_2^2,
\end{split}
\end{equation*}
\noindent and \eqref{Proximal IPM Penalty} is altered similarly. The second difference lies in the fact that we do not require the iterates of the method to lie in the neighbourhood defined in \eqref{Small neighbourhood} in order to gain efficiency. In what follows, we provide further details concerning our implementation choices.
\subsubsection{Free Variables}
\par The method takes as input problems in the following form:
\begin{equation*} 
\text{min}_{x} \ \big( c^Tx + \frac{1}{2}x^T Q x \big), \ \ \text{s.t.}  \  Ax = b,   \ x^{I} \geq 0,\ x^{F}\ \text{free},  
\end{equation*}
\noindent where $I = \{1,\cdots,n\} \setminus F$ is the set of indices indicating the non-negative variables. In particular, if a problem instance has only free variables, no logarithmic barrier is employed and the method reduces to a standard proximal method of multipliers. Of course in this case, the derived complexity result does not hold. Nevertheless, a global convergence result holds, as shown in \cite{paper_13}. In general, convex optimization problems with only equality constraints are usually easy to deal with, and the proposed algorithm behaves very well when solving such problems in practice.

\subsubsection{Constraint Matrix Scaling}

\par In the pre-processing stage, we check if the constraint matrix is well scaled, i.e if:
$$\big(\max_{i \in \{1,\ldots,m\},j \in \{1,\ldots,n\}}(|A^{ij}|) < 10\big) \wedge \big(\min_{i \in \{1,\ldots,m\},j \in \{1,\ldots,n\}:\ |A^{ij}| > 0}(|A^{ij}|) > 0.1\big).$$
\noindent If the previous is not satisfied, we apply geometric scaling in the rows of $A$, that is, we multiply each row of $A$ by a scalar of the form:
$$d^i = \frac{1}{\sqrt{\max_{j \in \{1,\cdots,n\}}(|A^{i,:}|) \cdot \min_{j \in \{1,\ldots,n\}:\ |A^{ij}| > 0}(|A^{i,:}|)}},\ \forall\ i \in \{1,\ldots,m\}$$
\par However, for numerical stability, we find the largest integer $p^i$, such that: $2^{p^i} \leq d^i$ and we set $d^i = 2^{p^i}$, $\forall\ i \in \{1,\ldots,m\}$. Hence, the scaling factors are powers of two. Based on the IEEE  representation of floating point numbers, multiplying by a power of 2 translates into an addition of this power to the exponent, without affecting the mantissa.

\subsubsection{Starting Point, Newton-step Computation and Step-length}
\par We use a starting point, based on the developments in \cite{paper_6}. To construct it, we try to solve the pair of problems (\ref{non-regularized primal})-(\ref{non-regularized dual}), ignoring the non-negativity constraints. 

\begin{equation*} 
\tilde{x} = A^T(AA^T)^{-1}b,\ \ \tilde{y} = (AA^T)^{-1}A(c+Q\tilde{x}), \ \ \tilde{z} = c - A^T \tilde{y} + Q\tilde{x}.
\end{equation*}

\par However, in order to ensure stability and efficiency, we regularize the matrix $AA^T$ and employ the Preconditioned Conjugate Gradient (PCG) method to solve these systems (in order to avoid forming $AA^T$). We use the classical Jacobi preconditioner to accelerate PCG, i.e. $P = \text{diag}(AA^T) + \delta I$, where $\delta = 8$, is set as the regularization parameter.
\par Then, to guarantee positivity and sufficient magnitude of $x_I,z_I$, we compute $\delta_x = \text{max}\{-1.5\cdot \text{min}(\tilde{x}^{I}),0\}$ and $\delta_z = \text{max}\{-1.5 \cdot \text{min}(\tilde{z}^{I}),0\}$ and we obtain:

\begin{equation*}
\tilde{\delta_x} = \delta_x + 0.5\frac{(\tilde{x}^I+ \delta_x e_{|I|})^T(\tilde{z}^I+ \delta_z e_{|I|})}{\sum_{i=1}^{|I|} (\tilde{z}^{I(i)} + \delta_z)}, \quad
\tilde{\delta_z} = \delta_z + 0.5\frac{(\tilde{x}^I+ \delta_x e_{|I|})^T(\tilde{z}^I+ \delta_z e_{|I|})}{\sum_{i=1}^{|I|} (\tilde{x}^{I(i)} + \delta_x)},
\end{equation*}
\noindent where $e_{|I|}$ is the vector of ones of appropriate dimension. Finally, we set:

\begin{equation*}
 y_0 = \tilde{y},\ \ z_0^{I} = \tilde{z}^{I} + \tilde{\delta_z}e_{|I|},\ \ z_0^{F} = 0,\ \ x_0^{I} = \tilde{x}^{I} + \tilde{\delta_x}e_{|I|},\ \ x_0^{F} = \tilde{x}^F.
\end{equation*}

\par In order to find the Newton step, we employ a widely used predictor-corrector method. The practical implementation deviates from the theory at this point, in order to gain computational efficiency. We provide the algorithmic scheme in Algorithm \ref{Algorithm Predictor-Corrector} for completeness, but the reader is referred to \cite{paper_6}, for an extensive review of the method. It is important to notice here that the centering parameter ($\sigma > 0$) is not present in Algorithm \ref{Algorithm Predictor-Corrector}. Solving two different systems, serves as a way of decreasing the infeasibility and allowing $\mu_k$ (and hence $\delta_k,\ \rho_k$) to decrease. 

\renewcommand{\thealgorithm}{PC}

\begin{algorithm}[!ht]
\caption{Predictor-Corrector method}
    \label{Algorithm Predictor-Corrector}
\begin{algorithmic}[]
	\State Compute the predictor:
	\begin{equation} \label{Augmented System predictor}
\begin{bmatrix} 
-(Q+\Theta_k^{-1}+\rho_k I) &   A^T \\
A & \delta_k I
\end{bmatrix}
\begin{bmatrix}
\Delta_p x\\ 
\Delta_p y
\end{bmatrix}
= 
\begin{bmatrix}
c + Qx_k - A^T y_k -  \rho_k (x_k - \zeta_k) - d_{1_k}\\
b-Ax_k - \delta_k (y_k - \lambda_k)
\end{bmatrix},
\end{equation}
\noindent where $d_{1_k}^I = -\mu_k (X_k^I)^{-1}e_{|I|}$, $d_{1_k}^F = 0$, and  $(\Theta_k^I)^{-1} = (X_k^I)^{-1} (Z_k^I)$, $(\Theta_k^F)^{-1} = 0$.
\State Retrieve $\Delta_p z$:
$$\Delta_p z^I = d_{1_k}^I - (X_k^I)^{-1}(Z_k^{I} \Delta_p x^{I}),\ \ \Delta_p z^F = 0.$$
\State Compute the step in the non-negativity orthant:
\begin{equation*} 
\alpha_x^{\max} = \text{min}_{(\Delta_p x_k^{I(i)} < 0)} \bigg \{1,-\frac{x^{I(i)}}{\Delta_p x^{I(i)}}\bigg\},\ \ \alpha_z^{\max} = \text{min}_{(\Delta_p z^{I(i)} < 0)} \bigg \{1,-\frac{z_k^{I(i)}}{\Delta_p z^{I(i)}}\bigg\},
\end{equation*}
\noindent for $i = 1,\cdots,|I|$, and set:
\begin{equation*}
\alpha_x = \tau \alpha_x^{\max},\ \ \alpha_z= \tau \alpha_z^{\max},
\end{equation*}
\noindent with $\tau = 0.995$ \Comment{avoid going too close to the boundary}.
\State Compute a centrality measure: 
$$g_{\alpha} = (x^I+\alpha_x \Delta_p x^I)^T(z^I + \alpha_z \Delta_p z^I).$$
\State Set: $\mu = \big(\frac{g_{\alpha}}{(x^I_k)^Tz^I_k} \big)^2 \frac{g_{\alpha}}{|I|}$
\State Compute the corrector:
\begin{equation} \label{Augmented System corrector}
\begin{bmatrix} 
-(Q+\Theta_k^{-1}+\rho_k I) &   A^T \\
A & \delta_k I
\end{bmatrix}
\begin{bmatrix}
\Delta_c x\\ 
\Delta_c y
\end{bmatrix}
= 
\begin{bmatrix}
 d_{2_k}\\
0
\end{bmatrix},
\end{equation}
\noindent with $d^I_{2_k} = \mu (X_k^I)^{-1}e_{|I|} - (X_k^I)^{-1}\Delta_p X^I  \Delta_p z^I$ and $d^F_{2_k} = 0$, $\Delta_p X = \text{diag}(\Delta_p x)$.
\State Retrieve $\Delta_c z$:
$$\Delta_c z^I = d_{2_k}^I - (X_k^I)^{-1}(Z_k^{I} \Delta_c x^{I}),\ \ \Delta_c z^F = 0.$$
\State $$(\Delta x, \Delta y, \Delta z) = (\Delta_p x + \Delta_c x, \Delta_p y + \Delta_c y, \Delta_p z + \Delta_c z).$$
\State Compute the step in the non-negativity orthant:
\begin{equation*} 
\alpha_x^{\max} = \text{min}_{\Delta x^{I(i)} < 0} \bigg \{1,-\frac{x_k^{I(i)}}{\Delta x^{I(i)}}\bigg\},\ \ \alpha_z^{\max} = \text{min}_{\Delta z^{I(i)} < 0} \bigg \{1,-\frac{z_k^{I(i)}}{\Delta z^{I(i)}}\bigg\},
\end{equation*}
\noindent and set:
\begin{equation*}
\alpha_x = \tau \alpha_x^{\max},\ \ \alpha_z= \tau\alpha_z^{\max}.
\end{equation*}
\State Update: 
$$(x_{k+1},y_{k+1},z_{k+1}) = (x_k+\alpha_x\Delta x,y_k + \alpha_z \Delta y,z_k + \alpha_z \Delta z).$$
\end{algorithmic}
\end{algorithm}
\par We solve the systems (\ref{Augmented System predictor}) and (\ref{Augmented System corrector}), using the built-in MATLAB symmetric decomposition (i.e. \texttt{ldl}). Such a decomposition always exists, with $D$ diagonal, for the aforementioned systems, since after introducing the regularization, the system matrices are guaranteed to be quasi-definite; a class of matrices known to be strongly factorizable, \cite{paper_4}. In order to exploit that, we change the default pivot threshold of \texttt{ldl} to a value slightly lower than the minimum allowed regularization value ($\text{reg}_{thr}$; specified in sub-section \ref{subsubsection regularization}). Using such a small pivot threshold guarantees that no 2x2 pivots will be employed during the factorization process.
\subsubsection{PMM Parameters} \label{subsubsection regularization}
\renewcommand{\thealgorithm}{PEU}
\par At this point, we discuss how we update the PMM sub-problems in practice, as well as how we tune the penalty parameters (regularization parameters) $\delta,\ \rho$. Notice that in Section \ref{section Polynomial Convergence} we set $\delta_k = \rho_k = \mu_k$. While this is beneficial in theory, as it gives us a reliable way of tuning the penalty parameters of the algorithm, it is not very beneficial in practice, as it does not allow us to control the regularization parameters in the cases of extreme ill-conditioning. Hence, we allow the use of different penalty parameters connected to the PMM sub-problems, while enforcing both parameters ($\delta_k,\ \rho_k$) to decrease at the same rate as $\mu_k$ (based on the theoretical results in Sections \ref{section Polynomial Convergence}, \ref{section Infeasible problems}).
\par On the other hand, the algorithm is more optimistic in practice than it is in theory. In particular, we do not consider the semi-norm \eqref{semi-norm definition} of the infeasibility, while we allow the update of the estimates $\lambda_k\ ,\zeta_k$, to happen independently. In particular, the former is updated when the 2-norm of the primal infeasibility is sufficiently reduced, while the latter is updated based on the dual infeasibility.
\par More specifically, at the beginning of the optimization, we set: $\delta_0 = 8,\ \rho_0 = 8$, $\lambda_0 = y_0$, $\zeta_0 = x_0$. Then, at the end of every iteration, we employ the algorithmic scheme given in Algorithm \ref{Algorithm Regularization updates}. In order to ensure numerical stability, we do not allow $\delta$ or $\rho$ to become smaller than a suitable positive value, $\text{reg}_{thr}$. We set: $\text{reg}_{thr} = \max\big\{\frac{\text{tol}}{\max\{\|A\|^2_{\infty},\|Q\|^2_{\infty}\}},10^{-10}\big\}$. This value is based on the developments in \cite{paper_35}, in order to ensure that we introduce a controlled perturbation in the eigenvalues of the non-regularized linear system. The reader is referred to \cite{paper_35} for an extensive study on the subject. If the factorization fails, we increase the regularization parameters by a factor of 10 and repeat the factorization. If the factorization fails while either $\delta$ or $\rho$ have reached their minimum allowed value ($\text{reg}_{thr}$), then we also increase this value by a factor of 10. If this occurs 5 consecutive times, the method is terminated with a message indicating ill-conditioning. 
\begin{algorithm}[!ht]
\caption{Penalty and Estimate Updates}
    \label{Algorithm Regularization updates}
\begin{algorithmic}[]
	\State $r = \frac{|\mu_{k}-\mu_{k+1}|}{\mu_k}$ (rate of decrease of $\mu$).
  	\If {($\|Ax_{k+1} - b\|_{2} \leq 0.95 \cdot \|Ax_k - b\|_2$)}
  		\State $\lambda_{k+1} = y_{k+1}$,
  		\State $\delta_{k+1} =(1-r) \cdot\delta_k$.
  	\Else
  		\State $\lambda_{k+1} = \lambda_k$,
  		\State $\delta_{k+1} = (1-\frac{1}{3}r) \cdot\delta_k$, \Comment{less aggressive is in this case.}
  	\EndIf
  	\State $\delta_{k+1} = \max \{\delta_{k+1},\text{reg}_{thr} \}$, \Comment{for numerical stability (ensure quasi-definiteness).}
  	\If {($\|c + Qx_{k+1} - A^Ty_{k+1} -z_{k+1}\|_2 \leq 0.95\cdot\|c+Qx_k - A^Ty_k - z_k\|_2$)}
  		\State $\zeta_{k+1} = x_{k+1}$.
  		\State $\rho_{k+1} = (1-r) \cdot\rho_k$.
  	\Else
  		\State $\zeta_{k+1} = \zeta_k$.
  		\State $\rho_{k+1} = (1-\frac{1}{3}r)\cdot \rho_k$.
  	\EndIf
    \State $\rho_{k+1} = \max \{\rho_{k+1},\text{reg}_{thr}\}$.
\end{algorithmic}
\end{algorithm}

\subsubsection{Termination Criteria}
\par There are four possible termination criteria. They are summarized in Algorithm \ref{Algorithm Termination Criteria}. In the aforementioned algorithm, tol represents the error tolerance chosen by the user. Similarly, $\text{IP}_{\text{maxit}}$ is the maximum allowed IPM iterations, also chosen by the user. On the other hand, $\text{PMM}_{\text{maxit}}$ is a threshold indicating that the PMM sub-problem needs too many iterations before being updated (that is if $k_{\text{PMM}} > \text{PMM}_{\text{maxit}}$). We set $\text{PMM}_{\text{maxit}} = 5$. When either $\lambda_k$ or $\zeta_k$ is updated, we set $k_{\text{PMM}} =0$. 
\par Let us now support the proposed infeasibility detection mechanism. In particular, notice that as long as the penalty parameters do not converge to zero, every PMM-subproblem must have a solution, even in the case of infeasibility. Hence, we expect convergence of the regularized primal (dual, respectively) infeasibility to zero, while from Section \ref{section Infeasible problems}, we know that a necessary condition for infeasibility is that the sequence $\|(x_k-\zeta_k,y_k - \lambda_k)\|_2$ diverges.  If this behaviour is observed, while the PMM parameters $\lambda_k,\ \zeta_k$ are not updated (which is not expected to happen in the feasible but rank deficient case), then we can conclude that the problem under consideration is infeasible.
\renewcommand{\thealgorithm}{TC}

\begin{algorithm}[!ht]
\caption{Termination Criteria}
    \label{Algorithm Termination Criteria}
\begin{algorithmic}[]
	\State \textbf{Input:} tol, $k_{\text{IP}}$, $k_{\text{PMM}}$, $\text{IP}_{\text{maxit}}$, $\text{PMM}_{\text{maxit}}$.
  	\If {\bigg($\big(\frac{\| c -A^T y_k + Qx_k - z_k \|_{2}}{\max\{\|c\|_2,1\}} \leq \text{tol}\big) \wedge \big(\frac{\|b - Ax_k\|_2}{\max\{\|b\|_2,1\}} \leq \text{tol} \big) \wedge \big(\mu_k \leq \text{tol}\big)$\bigg)}
  		\State \Return Solution $(x_k,y_k,z_k)$.
  	\ElsIf{\bigg($\big(\|c + Qx_k - A^T y_k - z_k + \rho_k(x_k - \zeta_k)\|_2 \leq \text{tol}\big) \wedge \big(\|x_k - \zeta_k\|_2 > 10^{10} \big)$\bigg)}
  		\If {$\big(k_{\text{PMM}}\geq \text{PMM}_{\text{maxit}} \big)$} \Comment{PMM sub-problem not updated for many iterations}
  			\State Declare Infeasibility.
  		\EndIf
  	\ElsIf{\bigg($\big(\|b-Ax_k - \delta_k(y_k-\lambda_k)\|_2 \leq \text{tol}\big)\ \wedge\ \big(\|y_k - \lambda_k\|_2> 10^{10}\big)$ \bigg)}
  	  	\If {$\big(k_{\text{PMM}}\geq \text{PMM}_{\text{maxit}} \big)$}
  			\State Declare Infeasibility.
  		\EndIf
  	\ElsIf{($k_{\text{IP}} \geq \text{IP}_{\text{maxit}}$)} \Comment{Maximum IPM iterations reached.}
  		\State No Convergence.
  	\EndIf

\end{algorithmic}
\end{algorithm}
\subsection{Numerical Results}

\par At this point, we present the computational results obtained by solving a set of small to large scale linear and convex quadratic problems. In order to stress out the importance of regularization, we compare IP-PMM with a non-regularized IPM. The latter implementation, follows exactly from the implementation of IP-PMM, simply by fixing $\delta$ and $\rho$ to zero. In the non-regularized case, the minimum pivot of the $\texttt{ldl}$ function is restored to its default value, in order to avoid numerical instability. Throughout all of the presented experiments, we set the number of maximum iterations to $200$. It should be noted here, that we expect IP-PMM to require more iterations to converge, as compared to the non-regularized IPM. In turn, the Newton systems arising in IP-PMM, have better numerical properties (accelerating the factorization process), while overall the method is expected to be significantly more stable. In what follows, we demonstrate that this increase in the number of iterations is benign, in that it does not make the resulting method inefficient. In contrast, we provide computational evidence that the acceleration of the factorization process more than compensates for the increase in the number of iterations.  The experiments were conducted on a PC with a 2.2GHz Intel Core i7 processor (hexa-core), 16GB RAM, run under Windows 10 operating system. The MATLAB version used was R2018b.

\subsubsection{Linear Programming Problems}
\par Let us compare the proposed method with the respective non-regularized implementation, over the Netlib collection, \cite{paper_7}. The test set consists of 96 linear programming problems. We set the desired tolerance to $\text{tol} = 10^{-6}$. Firstly, we compare the two methods, without using the pre-solved version of the problem collection (e.g. allowing rank-deficient matrices). In this case, the non-regularized IPM converged for only 66 out of the total 96 problems. On the other hand, IP-PMM solved successfully the whole set, in 160 seconds (and a total of 2,609 IPM iterations). Hence, one of the benefits of regularization, that of alleviating rank deficiency of the constraint matrix, becomes immediately obvious. 
\par However, in order to explore more potential benefits of regularization, we run the algorithm on a pre-solved Netlib library. In the pre-solved set, the non-regularized IPM converged for 93 out of 96 problems. The three failures occurred due to instability of the Newton system. The overall time spent was 353 seconds (and a total of 1,871 IPM iterations). On the other hand, IP-PMM solved the whole set in 161 seconds (and a total of 2,367 iterations). Two more benefits of regularization become obvious here. Firstly, we can observe that numerical stability can be a problem in a standard IPM, even if we ensure that the constraint matrices are of full row rank. Secondly, notice that despite the fact that IP-PMM required 26\% more iterations, it still solved the whole set in 55\% less CPU time. This is because in IP-PMM, only diagonal pivots are allowed during the factorization. We could enforce the same condition on the non-regularized IPM, but then significantly more problems would fail to converge (22/96) due to numerical instability. 
\par In Table \ref{Table netlib large instances}, we collect statistics from the runs of the two methods over some medium scale instances of the pre-solved Netlib test set. 
\begin{table}[H]
\centering
\caption{Medium-Scale Netlib Problems\label{Table netlib large instances}}
\begin{tabular}{@{\extracolsep{4pt}}rrrrrrr@{}}
    \toprule
\multirow{2}{*}{\textbf{Name}}            & \multirow{2}{*}{$\bm{nnz(A)}$}            & \multicolumn{2}{c}{{\textbf{IP-PMM}}}   & \multicolumn{2}{c}{\textbf{IPM}}\\  \cline{3-4}  \cline{5-6}
& &{Time (s)} & { IP-Iter.}  & {Time (s)}& {IP-Iter.} \\ \midrule
80BAU3B &   $29,063 $ & 1.43 & 44  & 9.68 & 40 \\
D6CUBE &  $43,888$ & 1.26 & 25  & 9.64 & 22  \\
DFL001 & $41,873$ & 25.42 & 47  & $\dagger$\tablefootnote{$\dagger$ means that the method did not converge.} & $\dagger$ \\ 
FIT2D & $138,018$ & 8.52 & 27  & 23.94 & 25  \\ 
FIT2P & $60,784$ & 1.24 & 24  & 1.56 & 19   \\ 
PILOT87 &  $73,804$ & 7.21 & 49  & 95.04 & 46 \\
QAP15 & $110,700$ & 91.78 & 23  & 93.56 & 18 \\ \bottomrule
\end{tabular}
\end{table}
\noindent From Table \ref{Table netlib large instances}, it becomes obvious that the factorization efficiency is significantly improved by the introduction of the regularization terms. In all of the presented instances, IP-PMM converged needing more iterations, but requiring less CPU time.
\par In order to summarize the comparison of the two methods, we include Figure \ref{figure LP perf. prof.}, which contains the performance profiles, over the pre-solved Netlib set, of the two methods. IP-PMM is represented by the green line (consisted of  triangles), while non-regularized IPM by the blue line (consisted of stars). In Figure \ref{subfigure LP perf. prof. time}, we present the performance profile with respect to time required for convergence, while in Figure \ref{subfigure LP perf. prof. iter}, the performance profile with respect to the number of iterations. In both figures, the horizontal axis is in logarithmic scale, and it represents the ratio with respect to the best performance achieved by one of the two methods, for each problem. The vertical axis shows the percentage of problems solved by each method, for different values of the performance ratio. Robustness is ``measured" by the maximum achievable percentage, while efficiency by the rate of increase of each of the lines (faster increase translates to better efficiency). For more information about performance profiles, we refer the reader to \cite{paper_27}, where this benchmarking was proposed. One can see that all of our previous observations are verified in Figure \ref{figure LP perf. prof.}.

\begin{figure}[h]
\centering
\caption{Performance profiles over the pre-solved Netlib test set.} \label{figure LP perf. prof.}
\begin{subfigure}[H]{0.4\textwidth}
	\includegraphics[width=\textwidth]{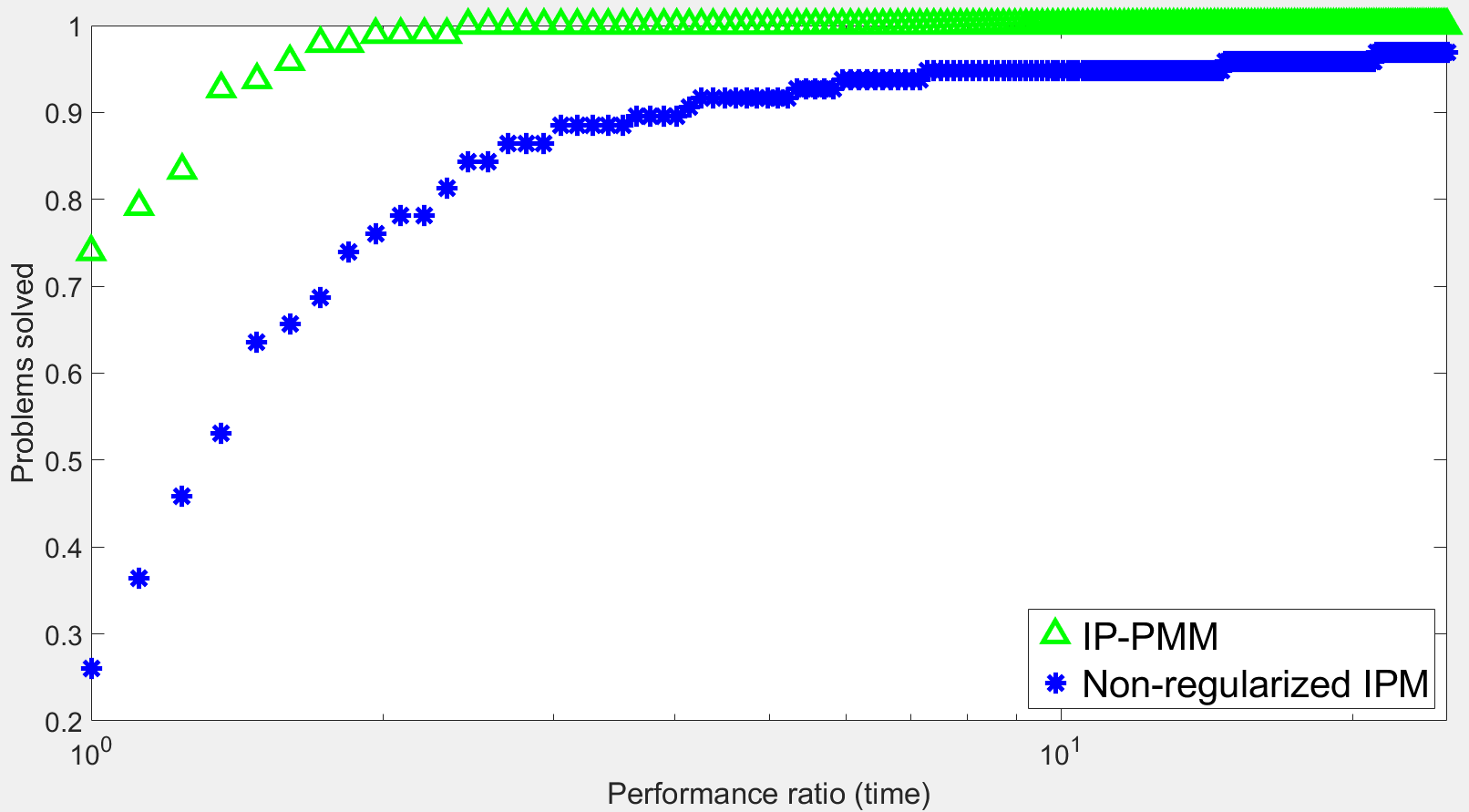}
	\caption{Performance profile - time.}
	\label{subfigure LP perf. prof. time}
\end{subfigure}
\quad
\begin{subfigure}[H]{0.4\textwidth}
	\includegraphics[width =\textwidth]{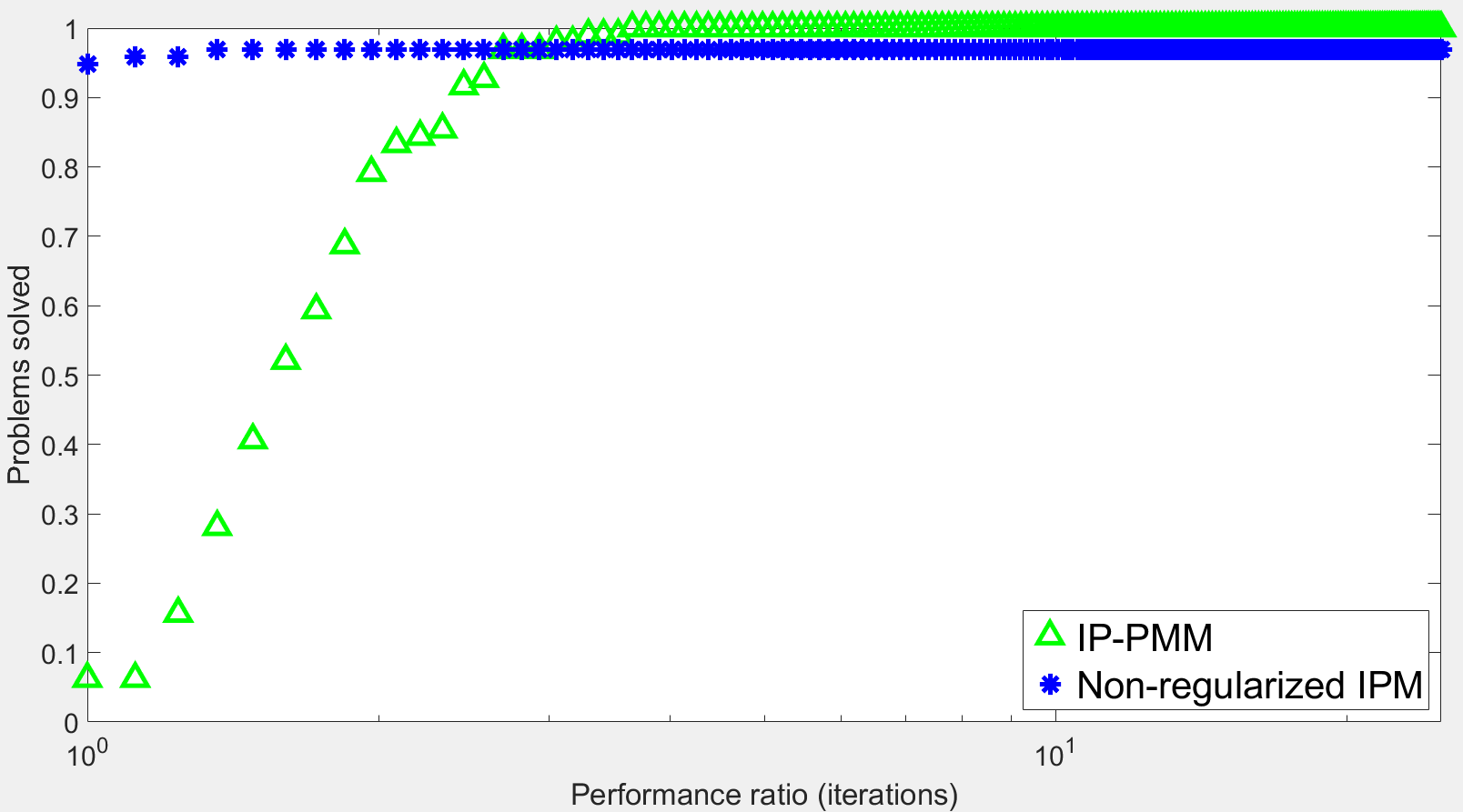}
	\caption{Performance profile - iterations.}
	\label{subfigure LP perf. prof. iter}
\end{subfigure}
\end{figure}

\subsubsection{Infeasible Problems}

\par In order to assess the accuracy of the proposed infeasibility detection criteria, we attempt to solve 28 infeasible problems, coming from the Netlib infeasible collection (\cite{paper_7}, see also \href{https://www.cise.ufl.edu/research/sparse/matrices/LPnetlib/index.html}{Infeasible Problems}). For 22 out of the 28 problems, the method was able to recognize that the problem under consideration is infeasible, and exit before the maximum number of iterations was reached. There were 4 problems, for which the method terminated after reaching the maximum number of iterations. For 1 problem the method was terminated early due to numerical instability. Finally, there was one problem for which the method converged to the least squares solution, which satisfied the optimality conditions for a tolerance of $10^{-6}$. Overall, IP-PMM run all 28 infeasible problems in 34 seconds (and a total of 1,813 IPM iterations). The proposed infeasibility detection mechanism had a 78\% rate of success over the infeasible test set, while no feasible problem was misclassified as infeasible. A more accurate infeasibility detection mechanism could be possible, however, the proposed approach is easy to implement and cheap from the computational point of view. Nevertheless, the interested reader is referred to \cite{paper_37,paper_42} and the references therein, for various other infeasibility detection methods.
 
\subsubsection{Quadratic Programming Problems}

\par Next, we present the comparison of the two methods over the Maros--M\'esz\'aros test set (\cite{paper_8}), which is comprised of 122 convex quadratic programming problems. Notice that in the previous experiments, we used the pre-solved version of the collection. However, we do not have a pre-solved version of this test set available. Since the focus of the paper is not on the pre-solve phase of convex problems, we present the comparison of the two methods over the set, without applying any pre-processing. As a consequence, non-regularized IPM fails to solve 27 out of the total 122 problems. However, only 11 out of 27 failed due to rank deficiency. The remaining 16 failures occurred due to numerical instability. On the contrary, IP-PMM solved the whole set successfully in 127 seconds (and a total of 3,014 iterations). As before, the required tolerance was set to $10^{-6}$.
\par In Table \ref{Table QP large instances}, we collect statistics from the runs of the two methods over some medium scale instances of the Maros--M\'esz\'aros collection. 

\begin{table}[!ht]
\centering
\caption{Medium-Scale Maros--M\'esz\'aros Problems\label{Table QP large instances}}
\begin{tabular}{@{\extracolsep{4pt}}rrrrrrrr@{}}
    \toprule
\multirow{2}{*}{\textbf{Name}}            & \multirow{2}{*}{$\bm{nnz(A)}$}            & \multirow{2}{*}{$\bm{nnz(Q)}$}& \multicolumn{2}{c}{{\textbf{IP-PMM}}}   & \multicolumn{2}{c}{\textbf{IPM}}\\  \cline{4-5}  \cline{6-7}
& & &{Time (s)} & { IP-Iter.}  & {Time (s)}& {IP-Iter.} \\ \midrule
AUG2DCQP &   $40,000 $ & $40,400$ & 4.70 & 83  & 7.21 & 111 \\
CVXQP1L &  $14,998$ & $69,968$ &25.54 & 38  &  $\dagger$ &  $\dagger$  \\
CVXQP3L & $22,497$ & $69,968$ & 45.69 & 59  & $\dagger$ & $\dagger$ \\ 
LISWET1 & $30,000$ & $10,002$ & 1.07 & 30  & 1.86 & 40  \\ 
POWELL20 & $20,000$ & $10,000$ &1.26 & 30  & 1.61 & 25   \\ 
QSHIP12L &  $16,170$ & $122,433$ &0.91 & 23  & $\dagger$ & $\dagger$ \\ 
STCQP1 & $13,338$ & $49,109$  & 0.38  &  16 &  6.89 &13 \\ \bottomrule
\end{tabular}
\end{table}
\par In order to summarize the comparison of the two methods, we include Figure \ref{figure QP perf. prof.}, which contains the performance profiles, over the Maros--M\'esz\'aros test set, of the two methods. IP-PMM is represented by the green line (consisted of  triangles), while non-regularized IPM by the blue line (consisted of stars). In Figure \ref{subfigure QP perf. prof. time}, we present the performance profile with respect to time required for convergence, while in Figure \ref{subfigure QP perf. prof. iter}, the performance profile with respect to the number of iterations.\\
\begin{figure}[h]
\centering
\caption{Performance profiles over the Maros--M\'esz\'aros test set.} \label{figure QP perf. prof.}
\begin{subfigure}[H]{0.4\textwidth}
	\includegraphics[width=\textwidth]{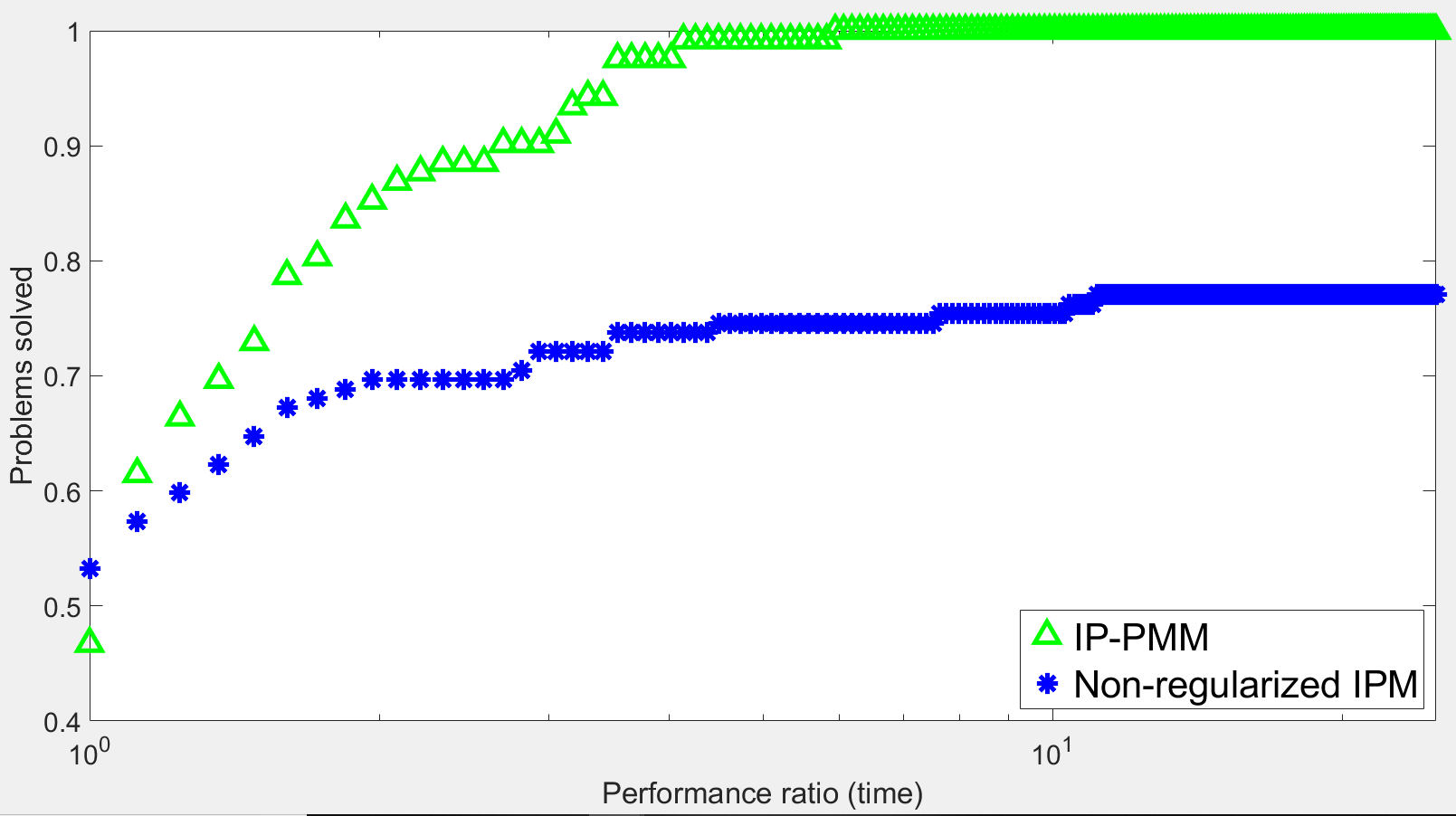}
	\caption{Performance profile - time.}
	\label{subfigure QP perf. prof. time}
\end{subfigure}
\quad
\begin{subfigure}[H]{0.4\textwidth}
	\includegraphics[width =\textwidth]{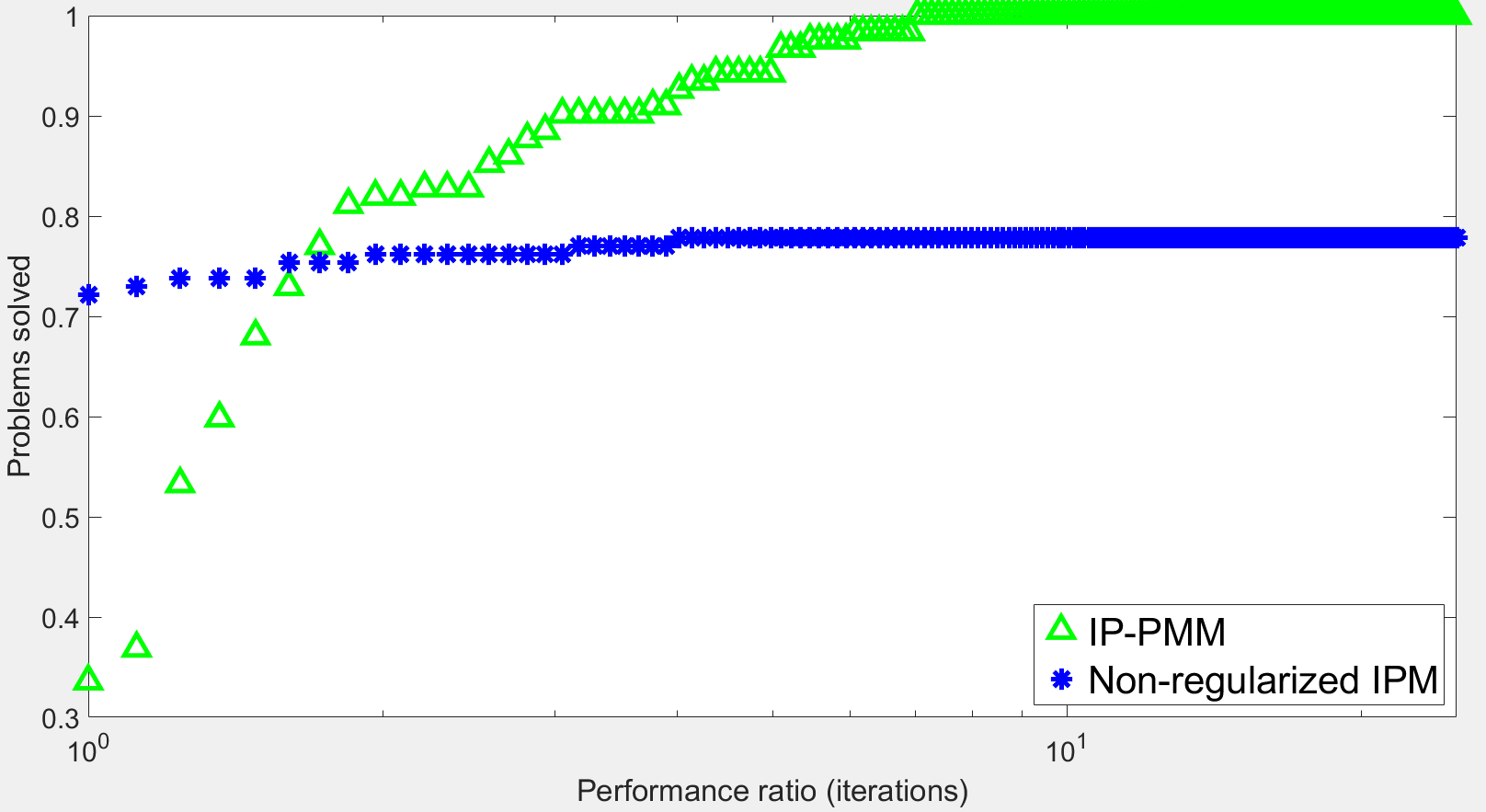}
	\caption{Performance profile - iterations.}
	\label{subfigure QP perf. prof. iter}
\end{subfigure}
\end{figure}

\noindent Similar remarks can be made here, as those given when summarizing the linear programming experiments. One can readily observe the importance of the stability introduced by the regularization. On the other hand, the overhead in terms of number of iterations that is introduced due to regularization, is acceptable due to the acceleration of the factorization (since we are guaranteed to have a quasi-definite augmented system).
\subsubsection{Verification of the Theory}

\par We have already presented the benefits of using regularization in interior point methods. A question arises, as to whether a regularized IPM can actually find an exact solution of the problem under consideration. Theoretically, we have proven this to be the case. However, in practice one is not allowed to decrease the regularization parameters indefinitely, since ill-conditioning will become a problem. Based on the theory of augmented Lagrangian methods, one knows that sufficiently small regularization parameters suffice for exactness (see \cite{book_6}, among others). In what follows, we provide a ``table of robustness" of IP-PMM. We run the method over the Netlib and the Maros-M\'esz\'aros collections, for decreasing values of the required tolerance and report the number of problems that converged.

\begin{table}[!ht]
\centering
\caption{Table of Robustness} \label{Table of robustness}
\begin{tabular}{@{\extracolsep{4pt}}rrr@{}}
    \toprule
\textbf{Test Set}          & \textbf{Tolerance}         & \textbf{Problems Converged} \\ \midrule
Netlib (non-presolved)  & $10^{-6}$ & 96/96   \\
"  & $10^{-8}$ & 95/96 \\
"  & $10^{-10}$ & 94/96 \\
Netlib (presolved)  & $10^{-6}$ & 96/96   \\
"  & $10^{-8}$ & 94/96 \\*
"  & $10^{-10}$ & 94/96 \\
 Maros-M\'esz\'aros  & $10^{-6}$ & 122/122   \\ 
" &  $10^{-8}$ &  121/122\\
" & $10^{-10}$ &  112/122\\ \bottomrule
\end{tabular}
\end{table}
\par One can observe from Table \ref{Table of robustness} that IP-PMM is sufficiently robust. Even in the case where a 10 digit accurate solution is required, the method is able to maintain a success rate larger than 91\%. 
\subsubsection{Large Scale Problems}

\par All of our previous experiments were conducted on small to medium scale linear and convex quadratic programming problems. We have showed (both theoretically and practically) that the proposed method is reliable. However, it is worth mentioning the limitations of the current approach. Since we employ exact factorization during the iterations of the IPM, we expect that the method will be limited in terms of the size of the problems it can solve. The main bottleneck arises from the factorization, which does not scale well in terms of processing time and memory requirements. In order to explore the limitations, in Table \ref{Table large scale} we provide the statistics of the runs of the method over a small set of large scale problems. It contains the number of non-zeros of the constraint matrices, as well as the time needed to solve the problem. The tolerance used in these experiments was $10^{-6}$.
\begin{table}[!ht]
\centering
\caption{Large-Scale Problems} \label{Table large scale}
\begin{tabular}{@{\extracolsep{4pt}}rrrrr@{}}
    \toprule
\textbf{Name}    & \textbf{Collection}        & $\bm{nnz(A)}$          & \textbf{time (s)} & \textbf{Status}\\ \midrule
fome21 & Mittelmann  & 751,365  & 567.26 & opt   \\
pds-10 & Mittelmann  & 250,307  & 40.00 & opt   \\
pds-30 & Mittelmann  & 564,988  & 447.81 & opt   \\
pds-60 & Mittelmann &  1,320,986 & 2,265.39 & opt   \\
pds-100 & Mittelmann & 1,953,757 & - & no memory   \\
rail582 & Mittelmann & 402,290 & 91.10 & opt   \\
cre-b & Kennington & 347,742  & 24.48 & opt   \\
cre-d & Kennington &  323,287  & 23.49 & opt   \\ 
stocfor3 & Kennington  & 72,721  & 4.56 & opt   \\
ken-18 & Kennington  & 667,569  & 77.94 & opt   \\
osa-30 & Kennington & 604,488 & 1723.96 & opt   \\ 
nug-20 & QAPLIB & 304,800 & 386.12 & opt   \\ 
nug-30 & QAPLIB  & 1,567,800 & - & no memory   \\ \bottomrule
\end{tabular}
\end{table}
\par From Table \ref{Table large scale}, it can be observed that as the dimension of the problem grows, the time to convergence is significantly increased. This increase in time is disproportionate for some problems. The reason for that, is that the required memory could exceed the available RAM, in which case the swap-file is activated. Access to the swap memory is extremely slow and hence the time could potentially increase disproportionately. On the other hand, we retrieve two failures due to lack of available memory. The previous issues could potentially be addressed by the use of iterative methods. Such methods, embedded in the IP-PMM framework, could significantly relax the memory as well as the processing requirements, at the expense of providing inexact directions. Combining IP-PMM (which is stable and reliable) with such an inexact scheme (which could accelerate the IPM iterations) seems to be a viable and competitive alternative and will be addressed in a future work. 
\section{Conclusions} \label{section conclusions}

\par In this paper, we present an algorithm suitable for solving convex quadratic programs. It arises from the combination of an infeasible interior point method with the proximal method of multipliers (IP-PMM). The method is interpreted as a primal-dual regularized IPM, and we prove that it is guaranteed to converge in a polynomial number of iterations, under standard assumptions. As the algorithm relies only on one penalty parameter, we use the well-known theory of IPMs to tune it. In particular, we treat this penalty as a barrier parameter, and hence the method is well-behaved independently of the problem under consideration. Additionally, we derive a necessary condition for infeasibility and use it to construct an infeasibility detection mechanism. The algorithm is implemented, and the reliability of the method is demonstrated. At the expense of some extra iterations, regularization improves the numerical properties of the interior point iterations. The increase in the number of iterations is benign, since factorization efficiency as well as stability is gained. Not only the method remains efficient, but it outperforms a similar non-regularized IPM scheme.
\par We observe the limitations of the current approach, due to the cost of factorization, and it is expected that embedding iterative methods in the current scheme might further improve the scalability of the algorithm at the expense of inexact directions. Since the reliability of IP-PMM is demonstrated, it only seems reasonable to allow for approximate Newton directions and still expect fast convergence. Hence, a future goal is to extend the theory as well as the implementation, in order to allow the use of iterative methods for solving the Newton system.

\section*{Acknowledgements}
\noindent SP acknowledges financial support from a Principal's Career Development PhD scholarship at the University of Edinburgh, as well as a scholarship from A. G. Leventis Foundation. JG acknowledges support from the EPSRC grant EP/N019652/1. JG and SP acknowledge support from the Google Faculty Research Award: ``Fast $(1+x)$-order methods for linear programming".
%\begin{acknowledgements}
%\end{acknowledgements}


\bibliography{references} 

\begin{thebibliography}{10}

\bibitem{paper_1}
A.~Altman and J.~Gondzio.
\newblock {Regularized {S}ymmetric {I}ndefinite {S}ystems in {I}nterior {P}oint
  {M}ethods for {L}inear and {Q}uadratic {O}ptimization}.
\newblock {\em Optim. Meth. and Soft.}, Vol. 11 \& 12$\ $:$\ $275--302, 1999.

\bibitem{paper_15}
P.~Armand and J.~Benoist.
\newblock {Uniform {B}oundedness of the {I}nverse of a {J}acobian {M}atrix
  {A}rising in {R}egularized {I}nterior-{P}oint {M}ethods}.
\newblock {\em Math. Prog.}, Vol. 137$\ $(No. 1\&2):$\ $587--592, 2013.

\bibitem{paper_17}
P.~Armand and R.~Omheni.
\newblock {A {M}ixed {L}ogarithmic {B}arrier-{A}ugmented {L}agrangian {M}ethod
  for {N}onlinear {O}ptimization}.
\newblock {\em J. Optim. Theory and Appl.}, Vol. 173$\ $:$\ $523--547, 2017.

\bibitem{paper_37}
P.~Armand and N.~N. Tran.
\newblock Rapid {I}nfeasibility {D}etection in a {M}ixed {L}ogarithmic
  {B}arrier-{A}ugmented {L}agrangian {M}ethod for {N}onlinear {O}ptimization.
\newblock {\em Optim. Meth. and Soft.}, Vol. 34$\ $(No. 5):$\ $991--1013, 2019.

\bibitem{paper_18}
S.~Arreckx and D.~Orban.
\newblock {A {R}egularized {F}actorization-{F}ree {M}ethod for
  {E}quality-{C}onstrained {O}ptimization}.
\newblock {\em SIAM J. Optim.}, Vol. 28$\ $(No. 2):$\ $1613--1639, 2018.

\bibitem{book_6}
P.~D. Bertsekas.
\newblock {\em {C}onstrained {O}ptimization and {L}agrange {M}ultiplier
  {M}ethods}.
\newblock Athena Scientific, 1996.

\bibitem{book_2}
P.~D. Bertsekas, A.~Nedic, and E.~Ozdaglar.
\newblock {\em {Convex {A}nalysis and {O}ptimization}}.
\newblock Athena Scientific, 2003.

\bibitem{book_5}
P.~D. Bertsekas and N.~J. Tsitsiklis.
\newblock {\em Parallel and {D}istributed {C}omputation: {N}umerical
  {M}ethods}.
\newblock Athena Scientific, 1997.

\bibitem{paper_20}
Y.~Censor and A.~Zenios.
\newblock {Proximal {M}inimization {A}lgorithm with {D}-{F}unctions}.
\newblock {\em J. of Optim. Theory and Appl.}, Vol. 73$\ $(No. 3):$\ $451--464,
  1992.

\bibitem{paper_27}
D.~E. Dolan and J.~J. Moré.
\newblock Benchmarking {O}ptimization {S}oftware with {P}erformance {P}rofiles.
\newblock {\em Math. Prog. Ser. A.}, Vol. 91$\ $:$\ $201--213, 2002.

\bibitem{paper_31}
J.~Eckstein.
\newblock Nonlinear {P}roximal {P}oint {A}lgorithms {U}sing {B}regman
  {F}unctions, with {A}pplications to {C}onvex {P}rogramming.
\newblock {\em Math. of Operational Research}, Vol. 18$\ $(No. 1), 1993.

\bibitem{paper_2}
M.~P. Friedlander and D.~Orban.
\newblock {A {P}rimal-{D}ual {R}egularized {I}nterior-{P}oint {M}ethod for
  {C}onvex {Q}uadratic {P}rograms}.
\newblock {\em Math. Prog. Comp.}, Vol. 4$\ $(No. 1):$\ $71--107, 2012.

\bibitem{paper_12}
O.~G$\ddot{\textnormal{u}}$ler.
\newblock {New {P}roximal {P}oint {A}lgorithms for {C}onvex {M}inimization}.
\newblock {\em SIAM J. Optim.}, Vol. 2$\ $(No. 4):$\ $649--664, 1992.

\bibitem{paper_36}
E.~P. Gill and P.~D. Robinson.
\newblock A {P}rimal-{D}ual {A}ugmented {L}agrangian.
\newblock {\em Comp. Optim. Appl.}, Vol. 51$\ $:$\ $1--25, 2012.

\bibitem{paper_28}
J.~Gondzio.
\newblock Multiple {C}entrality {C}orrections in a {P}rimal-{D}ual {M}ethod for
  {L}inear {P}rogramming.
\newblock {\em Comp. Optim. and Appl.}, Vol. 6$\ $:$\ $137--156, 1996.

\bibitem{paper_16}
J.~Gondzio.
\newblock {Interior {P}oint {M}ethods 25 {Y}ears {L}ater}.
\newblock {\em European J. of Operational Research}, Vol. 218$\ $(No. 3):$\
  $587--601, 2013.

\bibitem{paper_32}
M.~R. Hestenes.
\newblock Multiplier and {G}radient {M}ethods.
\newblock {\em J. of Optim. Theory and Appl.}, (No. 4):303--320, 1969.

\bibitem{paper_14}
A.~N. Iusem.
\newblock {Some {P}roperties of {G}eneralized {P}roximal {P}oint {M}ethods for
  {Q}uadratic and {L}inear {P}rogramming}.
\newblock {\em J. of Optim. Theory and Appl.}, Vol. 85$\ $(No. 3):$\ $593--612,
  1995.

\bibitem{paper_41}
M.~Kojima, N.~Megiddo, and S.~Mizuno.
\newblock {A} {P}rimal-{D}ual {I}nfeasible-{I}nterior-{P}oint {A}lgorithm for
  {L}inear {P}rogramming.
\newblock {\em Math. Prog.}, Vol. 61$\ $:$\ $263--280, 1993.

\bibitem{paper_34}
G.~Lan and R.~D.~C. Monteiro.
\newblock Iteration-{C}omplexity of {F}irst-{O}rder {A}ugmented {L}agrangian
  {M}ethods for {C}onvex {P}rogramming.
\newblock {\em Math. Prog.}, Vol. 155$\ $(No. 1-2):$\ $511--547, 2016.

\bibitem{paper_8}
I.~Maros and C.~M\'esz\'aros.
\newblock {A {R}epository of {C}onvex {Q}uadratic {P}rogramming {P}roblems}.
\newblock {\em Optim. Meth. and Soft.}, Vol. 11 \& 12$\ $:$\ $671--681, 1999.

\bibitem{paper_6}
S.~Mehrotra.
\newblock {On the {I}mplementation of a {P}rimal-{D}ual {I}nterior-{P}oint
  {M}ethod}.
\newblock {\em SIAM J. Optim.}, Vol. 2$\ $(No. 4):$\ $575--601, 1992.

\bibitem{paper_43}
S.~Mizuno and F.~Jarre.
\newblock {G}lobal and {P}olynomial-{T}ime {C}onvergence of an
  {I}nfeasible-{I}nterior-{P}oint {A}lgorithm using {I}nexact {C}omputation.
\newblock {\em Math. Prog.}, Vol. 84$\ $:$\ $105--122, 1999.

\bibitem{book_3}
Y.~Nesterov and A.~Nemirovski.
\newblock {\em {I}nterior-{P}oint {P}olynomial {A}lgorithms in {C}onvex
  {P}rogramming}.
\newblock SIAM, 1993.

\bibitem{paper_42}
Y.~Nesterov, M.~J. Todd, and Y.~Ye.
\newblock {I}nfeasible-{S}tart {P}rimal-{D}ual {M}ethods and {I}nfeasibility
  {D}etectors for {N}onlinear {P}rogramming {P}roblems.
\newblock {\em Math. Prog.}, Vol. 82$\ $(No. 2):$\ $227--267, 1999.

\bibitem{paper_7}
Netlib.
\newblock \url{http://netlib.org/lp}, 2011.

\bibitem{paper_21}
N.~Parikh and S.~Boyd.
\newblock {Proximal {A}lgorithms}.
\newblock {\em Foundations and Trends in Optim.}, Vol. 3$\ $(No. 1):$\
  $123--231, 2014.

\bibitem{paper_35}
S.~Pougkakiotis and J.~Gondzio.
\newblock Dynamic {N}on-diagonal {R}egularization in {I}nterior {P}oint
  {M}ethods for {L}inear and {C}onvex {Q}uadratic {P}rogramming.
\newblock {\em J. of Optim. Theory and Appl.}, Vol. 181$\ $(No. 3):$\
  $905--945, 2019.

\bibitem{paper_33}
M.~J.~D. Powell.
\newblock A {M}ethod for {N}onlinear {C}onstraints in {M}inimization
  {P}roblems.
\newblock {\em Optim.}, (No. 4):$\ $283--298, 1969.

\bibitem{paper_13}
R.~T. Rockafellar.
\newblock {Augmented {L}agrangians and {A}pplications of the {P}roximal {P}oint
  {A}lgorithm in {C}onvex {P}rogramming}.
\newblock {\em Math. of Operations Research}, Vol. 1$\ $(No. 2):$\ $97--116,
  1976.

\bibitem{paper_39}
R.~T. Rockafellar.
\newblock {M}onotone {O}perators and the {P}roximal {P}oint {A}lgorithm.
\newblock {\em SIAM J. Control and Optim.}, Vol. 14$\ $(No. 5), 1976.

\bibitem{book_7}
R.~T. Rockafellar.
\newblock {\em {C}onvex {A}nalysis}.
\newblock Princeton University Press, 1996.

\bibitem{paper_4}
R.~Vanderbei.
\newblock {Symmetric {Q}uasidefinite {M}atrices}.
\newblock {\em SIAM J. Optim.}, Vol. 5$\ $(No. 1):$\ $100--113, 1995.

\bibitem{paper_29}
Y.~Wang and P.~Fei.
\newblock An {I}nfeasible {M}izuno-{T}odd-{Y}e {T}ype {A}lgorithm for {C}onvex
  {Q}uadratic {P}rogramming with {P}olynomial {C}omplexity.
\newblock {\em Numer. Funct. Anal. and Optim.}, Vol. 28$\ $(No. 3-4):$\
  $487--502, 2007.

\bibitem{book_4}
S.~J. Wright.
\newblock {\em Primal-{D}ual {I}nterior {P}oint {M}ethods}.
\newblock SIAM, 1997.

\bibitem{paper_30}
Y.~Zhang.
\newblock On the {C}onvergence of a {C}lass of {I}nfeasible {I}nterior-{P}oint
  {M}ethods for the {H}orizontal {L}inear {C}omplementarity {P}roblem.
\newblock {\em SIAM J. Optim.}, Vol. 4$\ $(No. 1):$\ $208--227, 1994.

\bibitem{paper_38}
G.~Zhou, K.~C. Toh, and G.~Zhao.
\newblock Convergence {A}nalysis of an {I}nfeasible {I}nterior {P}oint
  {A}lgorithm {B}ased on a {R}egularized {C}entral {P}ath for {L}inear
  {C}omplementarity {P}roblems.
\newblock {\em Comp. Optim. Appl.}, Vol. 27$\ $(No. 3):$\ $269--283, 2004.

\end{thebibliography}
\bibliographystyle{abbrv}


\end{document}